\documentclass[11pt]{amsart}
\usepackage{a4wide,amssymb,amsmath,amsthm}
\usepackage[ansinew]{inputenc}

\vspace{2pt}

\theoremstyle{plain}
  \newtheorem{Th}{Theorem}
  \newtheorem*{Th*}{Theorem}

  \newtheorem{Pro}{Proposition}
\theoremstyle{definition}
  
  \newtheorem{defi}{Definition}
\theoremstyle{remark}
  \newtheorem*{Rm}{Remark}
\theoremstyle{remark}
  
\newcommand{\R}{\mathbb{R}}

\newcommand{\M}{M^n}

\newcommand{\Mcc}{M^2(c_1)\times M^2(c_2)}

\begin{document}

\title{CONSTANT ANGLE SURFACES IN PRODUCT SPACES}

\author[F. Dillen]{Franki Dillen}
\email[F. Dillen]{franki.dillen@wis.kuleuven.be}

\author[D. Kowalczyk]{Daniel Kowalczyk}
\email[D. Kowalczyk]{daniel.kowalczyk@wis.kuleuven.be}

\address{Katholieke Universiteit Leuven\\ Departement
Wiskunde\\ Celestij\-nenlaan 200 B, Box 2400\\ B-3001 Leuven\\ Belgium}

\thanks{This research was supported by Research Grant G.0432.07 of the Research Foundation-Flanders (FWO)}


\begin{abstract}
We classify all the surfaces in $M^2(c_1)\times M^2(c_2)$ for which the tangent space $T_pM^2$ makes constant angles with $T_p(M^2(c_1)\times \{p_2\})$ (or equivalently with $T_p(\{p_1\}\times M^2(c_2))$ for every point $p=(p_1,p_2)$ of $M^2$. Here $M^2(c_1)$ and $M^2(c_2)$ are $2$-dimensional space forms, not both flat. As a corollary we give a classification of all the totally geodesic surfaces in $M^2(c_1)\times M^2(c_2)$.
\end{abstract}

\maketitle

\section{Introduction}
In recent years a lot of people started the study of submanifolds in product spaces, in particular surfaces $M^2$ in $M^2(c)\times \mathbb{R}$, where $M^2(c)$ is a $2$-dimensional space form of curvature $c \neq 0$. This was initiated by the study of minimal surfaces in the product space $\mathbb{M}^2 \times \mathbb{R}$ by Meeks and Rosenberg in \cite{MR} and by Rosenberg in \cite{R}. In the papers \cite{DFVV} and \cite{DM} geometers began the study of constant angle surfaces in $M^2(c) \times \mathbb{R}$, i.e. surfaces for which the normal of the surface makes a constant angle with the vector field $\partial_t$ parallel to the second component of $M^2(c) \times \mathbb{R}$ and hence also with the first component $T_pM^2(c)$ of $T_p(M^2(c) \times \mathbb{R})$. They proved that they can construct all the constant angle surfaces in $M^2(c) \times \mathbb{R}$ starting from an arbitrary curve in $M^2(c)$ and that these surfaces have constant Gaussian curvature. Here we would like to define and classify constant angle surfaces in a  product space $\Mcc$ of two $2$-dimensional space forms, not both flat. We show that these constant angle surfaces have necessarily constant Gaussian curvature. In the classification theorem we show that some of the constant angle surfaces can be constructed from curves in $M^2(c_1)$ and $M^2(c_2)$. In other cases, the constant angle surfaces in $M^2(c_1)\times M^2(c_2)$ will be constructed from a solution of a Sine-(or Sinh-)Gordon equation and its B\"acklund transformation.

\section{Preliminaries}
\subsection{Surfaces in $\Mcc$.}
Let $M^2(c_1)\times M^2(c_2)$ be the product of two $2$-dimensional space forms of constant sectional curvature $c_1$ and $c_2$ with the standard product metric $\widetilde{g}$, with $c_1$ and $c_2$ not both $0$. Denote by $\widetilde{\nabla}$ the Levi-Civita connection of $(M^2(c_1)\times M^2(c_2),\widetilde{g})$ and by $F$ the product structure of $M^2(c_1)\times M^2(c_2)$, see \cite{YK}. This is the $(1,1)$-tensor of $M^2(c_1)\times M^2(c_2)$ defined by
\[
F(X_1 + X_2) = X_1 - X_2,\]
for any vector field $X = X_1 + X_2$, where $X_1$ and $X_2$ denote the parts of $X$ tangent to the first and second factors, respectively. By definition of the product structure $F$, we see that $\frac{I + F}{2}(X)$ is the projection of the vector field $X$ on the first component and that the $(1,1)$-tensor $\frac{I + F}{2}$ has rank $2$ everywhere. Analogously we have that $\frac{I - F}{2}(X)$ is the projection of the vector field $X$ on the second component and that the $(1,1)$-tensor $\frac{I - F}{2}$ has rank $2$ everywhere. We note that the product structure has the following properties:
\begin{gather}
F^2 = I\:(F \neq I)\label{F^2},\\
\widetilde{g}(FX,Y) = \widetilde{g}(X,FY)\label{Sym},
\end{gather}
and
\begin{equation*}
(\widetilde{\nabla}_XF)(Y)= 0,
\end{equation*}
for any vector field $X$ and $Y$ of $M^2(c_1)\times M^2(c_2)$. The Riemann-Christoffel curvature tensor $\widetilde{R}$ of $(M^2(c_1)\times M^2(c_2),\widetilde{g})$ is given by
\begin{equation*}
\widetilde{R}(X,Y)Z = c_1\left(\frac{I+F}{2}(X)\wedge \frac{I+F}{2}(Y)\right)Z + c_2\left(\frac{I-F}{2}(X)\wedge \frac{I-F}{2}(Y)\right)Z,
\end{equation*}
where $\wedge$ associates to two tangent vectors $v,w \in T_p(M^2(c_1)\times M^2(c_2))$ the endomorphism defined by
\[(v\wedge w)u = \widetilde{g}(w,u)v - \widetilde{g}(v,u)w,\]
for every $u \in T_p(M^2(c_1)\times M^2(c_2))$.

Let us now consider a surface $M^2$ immersed in $M^2(c_1)\times M^2(c_2)$. We will denote by $X,Y,Z,\dots$ tangent vector fields and by $\xi,\xi_1,\xi_2,\dots$ vector fields normal to $M^2$ in $M^2(c_1)\times M^2(c_2)$. We can now let the product structure $F$ of $M^2(c_1)\times M^2(c_2)$ act on a tangent vector field $X$ or on a normal vector field $\xi$. We can consider the decomposition of $FX$ and $F\xi$ into a tangent component and a normal component as
\begin{gather*}
FX = fX + hX,\\
F\xi = s\xi + t\xi,
\end{gather*}
where $f: TM^2 \rightarrow TM^2$, $h: TM^2 \rightarrow T^{\perp}M^2$, $s: T^{\perp}M^2 \rightarrow TM^2$ and $t: T^{\perp}M^2 \rightarrow T^{\perp}M^2$  are  $(1,1)$-tensors on $M^2$. It can be easily deduced from equations $(\ref{F^2})$ and $ (\ref{Sym})$ that
\begin{gather}
f \textrm{ is a symmetric } (1,1)\textrm{-tensor field on }M^2\textrm{ such that } f^2X = X - shX,\label{symf}\\
t \textrm{ is a symmetric } (1,1)\textrm{-tensor field on }M^2\textrm{ such that } t^2\xi = \xi - hs\xi,\label{symt}\\
\widetilde{g}(hX,\xi) = g(X,s\xi),\label{trans}\\
fs\xi + st\xi = 0 \textrm{ and } hfX + thX = 0\label{orth},
\end{gather}
for every $X \in TM^2$ and every $\xi \in T^{\perp}M^2$. If we denote by $R$ the Riemann-Christoffel curvature tensor of $M^2$, then with the previous notations we obtain that Gauss, Codazzi and Ricci equations are written as follows in terms of $f$ and $h$:
\begin{multline}\label{Gauss}
R(X,Y)Z = S_{\sigma(Y,Z)}X - S_{\sigma(X,Z)}Y + a((X \wedge Y)Z + \\ (fX \wedge fY)Z) + b(f(X \wedge Y)Z + (X \wedge Y)fZ),
\end{multline}
\begin{multline}\label{Codazzi}
(\nabla\sigma)(X,Y,Z) - (\nabla\sigma)(Y,X,Z) = a(g(fY,Z)hX - g(fX,Z)hY) +\\ b(g(Y,Z)hX - g(X,Z)hY),
\end{multline}
\begin{equation}\label{Ricci}
R^{\perp}(X,Y)\xi =a(\widetilde{g}(hY,\xi)hX - \widetilde{g}(hX,\xi)hY) - \sigma(S_{\xi}X,Y) + \sigma(S_{\xi}Y,X),
\end{equation}
where $a = \frac{c_1 + c_2}{4}$ and $b = \frac{c_1 - c_2}{4}$. Since $M^2$ is a surface immersed in $M^2(c_1)\times M^2(c_2)$, we have that equation ($\ref{Gauss}$) is equivalent to the fact that the Gaussian curvature $K$ is equal to
\begin{equation*}
\det(S_{\xi_1}) + \det(S_{\xi_2}) + c_1\det(\frac{I + f}{2}) + c_2\det(\frac{I - f}{2}),
\end{equation*}
where $\{\xi_1,\xi_2\}$ is an orthonormal basis of $T^{\perp}M^2$. Moreover we have the following proposition that we can prove using the formulas of Gauss and Weingarten and the fact that $\widetilde{\nabla}F = 0$.
\begin{Pro}
For every $X,Y \in TM^2$ and every $\xi \in T^{\perp}M^2$, we have that
\begin{gather}
(\nabla_X f)(Y) = S_{hY}X + s(\sigma(X,Y)),\label{parf}\\
\nabla^{\perp}_X hY - h(\nabla_X Y) = t(\sigma(X,Y)) - \sigma(X,fY),\label{parh}\\
\nabla^{\perp}_X t\xi - t(\nabla^{\perp}_X \xi) = -\sigma(s\xi,X) - h(S_{\xi}X),\label{part}\\
\nabla_X s\xi - s(\nabla^{\perp}_X\xi) = -fS_{\xi}X + S_{t\xi}X.\label{pars}
\end{gather}
\end{Pro}

We remark that the $(1,1)$-tensor $s$ is, in some sense, a kind of transpose of the $(1,1)$-tensor $h$, because of equation (\ref{trans}). We can easily see that equations (\ref{parh}) and (\ref{pars}) are equivalent because of equation (\ref{trans}). Analogously we can see that the two equations of (\ref{orth}) are equivalent.

The equations (\ref{Gauss}), (\ref{Codazzi}), (\ref{Ricci}), (\ref{parf}), (\ref{parh}) and (\ref{part}) are called the compatibility equations of surfaces in $M^2(c_1)\times M^2(c_2)$. The following theorems follow from more general results proven in \cite{K}.

\begin{Th}
Let $(M^2,g)$ be a simply connected Riemannian surface with Levi-Civita connection $\nabla$, $\nu$ a Riemannian vector bundle over $M^2$ of rank $2$ with metric $\widetilde{g}$, $\nabla^{\perp}$ a connection on $\nu$ compatible with the metric $\widetilde{g}$, $\sigma$ a symmetric $(1,2)$ tensor with values in $\nu$. Let $f:TM^2 \rightarrow TM^2$ , $t:\nu \rightarrow \nu$ and $h: TM^2 \rightarrow \nu$  be $(1,1)$-tensors over $M^2$ that satisfy equations (\ref{symf}), (\ref{symt}) and (\ref{orth}). Define $s:\nu \rightarrow TM^2$ by $g(s\xi,X) = \widetilde{g}(\xi,hX)$  for $X,Y \in TM^2$ , $\xi \in \nu$. Moreover $\frac{I + F}{2}$ and $\frac{I - F}{2}$ are bundle maps of rank $2$ defined such that $FX = fX + hX$ and $F\xi = s\xi + t\xi$. Assume that the compatibility equations for $M^2(c_1)\times M^2(c_2)$ are satisfied. Then there exists an isometric immersion $\psi:M^2 \rightarrow M^2(c_1)\times M^2(c_2)$ such that $\sigma$ is the second fundamental form, $\nu$ is isomorphic to the normal bundle of $\psi(M^2)$ in $M^2(c_1)\times M^2(c_2)$ by an isomorphism $\widetilde{\psi}: \nu \rightarrow T^{\perp}\psi(M^2)$ and such that
\begin{equation}\label{structure}
\widetilde{F}(\psi_{*}X) = \psi_{*}(fX) + \widetilde{\psi}(hX),
\end{equation}
and
\begin{equation}\label{structure2}
\widetilde{F}(\widetilde{\psi}\xi) = \psi_{*}(s\xi) + \widetilde{\psi}(t\xi),
\end{equation}
where $\widetilde{F}$ is the product structure of $M^2(c_1)\times M^2(c_2)$.
\end{Th}

\begin{Th}
Let $\psi: M^2 \rightarrow M^2(c_1)\times M^2(c_2)$, resp. $\psi':M^2\rightarrow M^2(c_1)\times M^2(c_2)$, be isometric immersions, with corresponding second fundamental form $\sigma$, resp. $\sigma'$, shape operator $S$, resp. $S'$, normal space $T^{\perp}M^2$, resp. $T^{\perp'}M^2$. Let $f$ and $h$ be $(1,1)$-tensors on $M^2$ defined by (\ref{structure}) and $f'$ and $h'$ similarly for $\psi'$. Suppose that the following conditions hold:
\begin{enumerate}
  \item $fX = f'X$ for every $X \in T_pM^2$ and $p \in M^2$.
  \item There exists an isometric bundle map $\widetilde{\phi}:T^{\perp}M^2 \rightarrow T^{\perp'}M^2$ such that
  \begin{gather*}
  \widetilde{\phi}(\sigma(X,Y)) = \sigma'(X,Y),\\
  \widetilde{\phi}(\nabla^{\perp}_X\xi) = \nabla^{\perp'}_{X}\widetilde{\phi}(\xi)
  \end{gather*}
  and
  \begin{equation*}
  \widetilde{\phi}(hX) = h'X
  \end{equation*}
  for every $X \in T_pM^2$, $\xi \in T^{\perp}_pM^2$ and $p \in M^2$.
\end{enumerate}
Then there exists an isometry $\tau$ of $M^2(c_1)\times M^2(c_2)$ such that $\tau \circ \psi = \psi'$ and $\tau_{*|T^{\perp}\M} = \widetilde{\phi}$.
\end{Th}

\subsection{Curves in $M^2(c)$.}

In this short subsection we will discuss curves in $2$-dimensional space forms $M^2(c)$ with $c \neq 0$. It is known that $M^2(c)$ is isometric to the $2$-dimensional sphere $\mathbb{S}^2(c)$ of radius $\frac{1}{\sqrt{c}}$ if $c > 0$, i.e.
\[
\mathbb{S}^2(c) = \{(p_1,p_2,p_3) \in \mathbb{E}^3\;|\; p_1^2 + p_2^2 + p_3^2 = \frac{1}{c}\},
\]
endowed with the induced metric of $\mathbb{E}^3$. The tangent space $T_p\mathbb{S}^2(c)$ in every point $p$ is given by
\[
\mathbb{S}^2(c) = \{v \in T_p\mathbb{E}^3\;|\; \langle p,v\rangle = 0\}.
\]
Using the cross-product $\times$ in $\mathbb{E}^3$, we define a complex structure $J$ on $T\mathbb{S}^2(c)$ by
\[J: T\mathbb{S}^2(c) \rightarrow T\mathbb{S}^2(c): v_p \mapsto \sqrt{c}(p\times v)_p.\]
It is easy to see that if $v \in T_p\mathbb{S}^2(c)$ and $\|v\|^2 = 1$, then $\{v,Jv\}$ is an orthonormal basis of $T_p\mathbb{S}^2(c)$.

We can define in a similar manner a complex structure when $c < 0$. It is known that $M^2(c)$ is isometric to the hyperbolic plane $\mathbb{H}^2(c)$ if $c < 0$. We use here the Minkowski or the hyperboloid model of the hyperbolic plane. Denote by $\mathbb{R}^3_1$ the Minkowski $3$-space with standard coordinates $p_1, p_2$ and $p_3$, endowed with the Lorentzian metric
\[
\langle .,.\rangle_1 = -dp_1^2 + dp_2^2 + dp_3^2.
\]
The hyperbolic plane $\mathbb{H}^2(c)$ can be constructed as the upper sheet ($p_1 > 0$) of the hyperboloid
\[
\{(p_1,p_2,p_3) \in \mathbb{R}^3_1\;|\; -p_1^2 + p_2^2 + p_3^2 = \frac{1}{c}\},
\]
endowed with the induced metric of $\mathbb{R}^3_1$. The tangent space $T_p\mathbb{H}^2(c)$ in every point $p$ is given by
\[
T_p\mathbb{H}^2(c) = \{v \in T_p\mathbb{R}^3_1\;|\; \langle p,v\rangle_1 = 0\}.
\]
Using the Lorentzian cross-product $\boxtimes$ in $\mathbb{R}^3_1$ (see for example \cite{DM}), we define a complex structure $J$ on $T\mathbb{H}^2(c)$ by
\[J: T\mathbb{H}^2(c) \rightarrow T\mathbb{H}^2(c): v_p \mapsto \sqrt{-c}(p\boxtimes v)_p.\]
It is easy to see that if $v \in T_p\mathbb{H}^2(c)$ and $\|v\|^2 = 1$, then $\{v,Jv\}$ is an orthonormal basis of $T_p\mathbb{H}^2(c)$. In the following we will denote $J$ as the complex structure of $M^2(c)$.

Let $\alpha:I \rightarrow M^2(c)$ be an arc-length  parameterized curve in $M^2(c)$. Denote by $T(s) \in T_{\alpha(s)}M^2(c) $ the tangent unit vector $\alpha'(s)$ and by $N(s) \in T_{\alpha(s)}M^2(c)$ the normal vector $JT(s)$. By direct calculations, one can show that
\begin{gather*}
T' = D_{T}T = \kappa N - c\alpha,
N' = D_{T}N = -\kappa T,
\end{gather*}
where $D$ is the Levi-Civita connection of $\mathbb{E}^3$ or of $\mathbb{R}^3_1$. We call $\kappa$ the geodesic curvature of $\alpha$ in $M^2(c)$. We will need the geodesic curvature of a curve in $M^2(c)$ in order to state our classification results of constant angle surfaces.

\section{Constant angle surfaces}

Since $f$ is a symmetric $(1,1)$-tensor on $M^2$, there exist continuous functions $\lambda_1 \leq \lambda_2$ on $M^2$ such that for every $p$ in $M^2$ $\lambda_1(p)$ and $\lambda_2(p)$ are eigenvalues of $f$ at $p$. Moreover $\lambda_1$ and $\lambda_2$ are differentiable functions in points where $\lambda_1$ and $\lambda_2$ are different. Assume that $\lambda_1 < \lambda_2$, then one can show that the distributions $T_{\lambda_1} =\{X \in TM^2 | fX = \lambda_1 X\}$ and $T_{\lambda_2} =\{X \in TM^2 | fX = \lambda_2 X\}$ are differentiable. From equations (\ref{symf}) and (\ref{trans}) it is easy to deduce that $\lambda_i^2 \leq 1$ for $i =1,2$. Hence we have that for every point $p$ there exists a unique $\theta_1(p)$ and  $\theta_2(p)$ in $[0,\frac{\pi}{2}]$ such that
\[\lambda_1(p) = \cos(2\theta_1(p))\qquad \textrm{and}\qquad \lambda_2(p) = \cos(2\theta_2(p)).\]
We call $\theta_1$ and $\theta_2$ the angle functions of $M^2$ in $M^2(c_1)\times M^2(c_2)$. This definition is inspired by the definition of angles between $2$-dimensional linear subspaces of the Euclidean space $\mathbb{E}^4$ given in \cite{MR0482513} or \cite{J}, where $\theta_1(p)$ and $\theta_2(p)$ are the angles between $T_pM^2$ and $T_p(M^2(c_1)\times \{p_2\})$, $p=(p_1,p_2)\in M$. Moreover this definition of angle for surfaces in $M^2(c_1)\times M^2(c_2)$ coincides with the definition of angle for surfaces in $M^2(c)\times \R$.

For Lagrangian surfaces in $S^2\times S^2$, a similar notion for angle was introduced in \cite{KS}; since for Lagrangian surfaces $\lambda_1+\lambda_2=0$, see below, there is only one angle function. Lagrangian surfaces in $S^2\times S^2$ are also studied in \cite{CU}.

\begin{defi}
A surface in $M^2(c_1)\times M^2(c_2)$ is a constant angle surface if $\theta_1$ and $\theta_2$ are constant.
\end{defi}

This definition also makes sense for $c_1=c_2=0$, but in this case it is better to call a surface in $\mathbb{E}^4$ a constant angle surface if there is a fixed plane in $\mathbb{E}^4$ such that $T_pM$ makes constant angles with this plane. This will be studied in a separate paper. Under additional conditions, a classification of those surfaces independently have been classified in \cite{BDOR}.

\subsection{Complex structures}

Let $\widetilde{J}$ and $\overline{J}$ be complex structures on $M^2(c_1)\times M^2(c_2)$ defined by
\[\widetilde{J}v = \widetilde{J}(v_1,v_2) = (J_1v_1,J_2v_2) =  (J_1\frac{I + F}{2}v,J_2\frac{I - F}{2}v)\]
and
\[\overline{J}v = \overline{J}(v_1,v_2) = (J_1v_1,-J_2v_2) =  (J_1\frac{I + F}{2}v,-J_2\frac{I - F}{2}v),\]
respectively, where $J_1$ and $J_2$ denote the standard complex structures on $M^2(c_1)$ and $M^2(c_2)$. We obtain the following connection between the angle functions and the complex structures $\widetilde{J}$ and $\overline{J}$.
\begin{Pro}
Consider a surface $M^2$ in $M^2(c_1)\times M^2(c_2)$ with angle functions $\theta_1$ and $\theta_2$, then
\begin{equation}
g(\widetilde{J}v,w) = \cos(\theta_1 - \theta_2)\omega_{M^2}(v,w) \ \  \mathrm{ or }\ \ \cos(\theta_1 + \theta_2)\omega_{M^2}(v,w)
\end{equation}
and
\begin{equation}
g(\overline{J}v,w) = \cos(\theta_1 + \theta_2)\omega_{M^2}(v,w)\ \ \mathrm{ or }\ \ \cos(\theta_1 - \theta_2)\omega_{M^2}(v,w),
\end{equation}
for all $v,w \in T_pM^2$ and $p\in M$ and a suitable choice of volume form $\omega_{M^2}$ of $M^2$.
\end{Pro}
\begin{proof}
We only prove this proposition in the case that $c_1,c_2 > 0$. The other cases can be proved analogously. Let us consider an orthonormal basis $\{e_1,e_2\}$ of $T_pM^2$ that diagonalizes $f$. Hence we have that $fe_i = \cos(2\theta_i)e_i$ for $i = 1,2$. Then
\[\begin{split}
g(\widetilde{J}e_1,e_2)& = -\sqrt{c_1}(\frac{I + F}{2}e_1\times \frac{I + F}{2}e_2)\cdot p_1 -\sqrt{c_2}(\frac{I - F}{2}e_1\times \frac{I - F}{2}e_2)\cdot p_2\\
& = \epsilon_1 (\|\frac{I + F}{2}e_1\|\|\frac{I + F}{2}e_2\| +\epsilon_2  \|\frac{I - F}{2}e_1\|\|\frac{I - F}{2}e_2\|)\\
& = \epsilon_1\cos(\theta_1 - \theta_2)\ \ \mathrm{ or }  \ \   \epsilon_1 \cos(\theta_1 + \theta_2),
\end{split}\]
where $\epsilon_1^2=\epsilon_2^2=1$. This proves the first equation in the proposition. The second equation can be proved similarly and the proof of the proposition is finished.
\end{proof}

By direct computations, one can now easily prove the following theorem.

\begin{Th}
A surface $M$ in $M^2(c_1)\times M^2(c_2)$ is a complex surface with respect to $\widetilde{J}$ or $\overline{J}$ if and only if $f$ is proportional to the identity. $M^2$ is Lagrangian with respect to $\widetilde{J}$ or $\overline{J}$ if and only if the trace of $f$ vanishes.
\end{Th}

\subsection{Totally geodesic surfaces}

We show now that totally geodesic surfaces in $M^2(c_1)\times M^2(c_2)$ are constant angle surfaces in $M^2(c_1)\times M^2(c_2)$.
\begin{Pro}\label{geodconstant}
Suppose $M^2$ is a totally geodesic surface of $M^2(c_1)\times M^2(c_2)$, then $M^2$ is a constant angle surface in $M^2(c_1)\times M^2(c_2)$.
\end{Pro}
\begin{proof}
As $M^2$ is a totally geodesic surface, we have that $(\nabla_X f) = 0$ for any $X \in TM^2$ and hence the eigenvalues of $f$ are constant. We give also the explicit values of $\lambda_1$ and $\lambda_2$, because we will need these values in the classification of totally geodesic surfaces. Let $p$ be an arbitrary point in $M^2(c_1)\times M^2(c_2)$ and $\{e_1,e_2\}$ an orthonormal basis in $T_pM^2$ such that $fe_1 = \lambda_1e_1$ and $fe_2 = \lambda_2e_2$. Using the equation of Codazzi, we obtain
\[
0 = (a\lambda_1 + b)(1 - \lambda^2_2) = (a\lambda_2 + b)(1 - \lambda^2_1),
\]
with $a = \frac{c_1 + c_2}{4}$ and $b = \frac{c_1 - c_2}{4}$. Hence we obtain that $\lambda_1 = \lambda_2 = -\frac{b}{a}$ with $c_1c_2 > 0$, $\lambda^2_1 = \lambda^2_2 = 1$, $\lambda^2_1 = 1$ and $\lambda_2 = -\frac{b}{a}$ with $c_1c_2 > 0$ or $\lambda_1^2 = 1$ and $\lambda_2 = \cos(2\theta)$ with $c_1c_2 = 0$ and $\theta \in [0,\frac{\pi}{2}]$. These conditions are equivalent to
\begin{enumerate}
               \item $\lambda_1 = \lambda_2 = -\frac{b}{a}$ with $c_1c_2 > 0 $,
               \item $\lambda_1= \pm 1$ and $\lambda_2 = \pm 1$,
               \item $\lambda_1 = \pm 1$ and $\lambda_2 = -\frac{b}{a}$ with $c_1c_2 > 0$, and
               \item $\lambda_1 = 1$ and $\lambda_2 \in [-1,1]$ with $c_1 = 0$  or $\lambda_1 = -1$ and $\lambda_2 \in [-1,1]$ with $c_2 = 0$.
\end{enumerate}
Since $\lambda_1$ and $\lambda_2$ are continuous, one of the above conditions must hold. Hence we obtain that totally geodesic surfaces of $M^2(c_1)\times M^2(c_2)$ are constant angle surfaces,  in which $\lambda_1$ and $\lambda_2$ have one of the above specific values.
\end{proof}
In this section we will give a local classification of totally geodesic surfaces for which $\lambda_1 = \lambda_2 = -\frac{b}{a}$ with $c_1c_2 > 0$. The other cases will be treated in the next sections and will appear as special cases of constant angle surfaces. We classify the totally geodesic surfaces of $M^2(c_1)\times M^2(c_2)$ in Theorem $\ref{theoremgeod}$. Suppose that $c_1,c_2 >0$, the other case can be treated analogously and the result of the second case is stated together with the first case in Proposition \ref{pro:geod}.

We can immerse $M^2(c_1)\times M^2(c_2)$ as a submanifold of codimension $2$ in the Euclidean space $\mathbb{E}^6$. We also remark that we obtain, by using the equation of Gauss, that the surface $M^2$ has constant Gaussian curvature $\frac{c_1c_2}{c_1 + c_2}$. So let $\psi:M^2 \rightarrow M^2(c_1)\times M^2(c_2)$ be a totally geodesic surface in $M^2(c_1)\times M^2(c_2)$ with $\lambda_1 = \lambda_2 = -\frac{b}{a}$. Let us fix a point $p$ in open set $U$ of $M^2$ and let $(u,v)$ be Fermi coordinates of $U$ in $M^2$, there always exist such coordinates on an open set of a surface (see for example \cite{Kuh}). The metric $g$ of $M$ then has the form
\[du^2 + G(u,v)dv^2\]
on the open set $U$ of $M$, with $G(0,v) = 1$ and $\frac{\partial}{\partial_u}G(0,v) = 0$ for every $v$, in terms of the Fermi coordinates $(u,v)$. Since $M^2$ has constant Gaussian curvature, we have that $G$ is uniquely determined by the partial differential equation
\[\frac{\partial^2}{\partial^2u}\sqrt{G} = -\frac{c_1c_2}{c_1 + c_2}\sqrt{G}.\]
So we find that $G$ is given by
\begin{equation}\label{metricG}
    \cos^2(\sqrt{\frac{c_1c_2}{c_1 + c_2}}u),
\end{equation}
because of the initial conditions $G(0,v) = 1$ and $\frac{\partial}{\partial_u}G(0,v) = 0$ for every $v$. Let us now consider $M^2$ as a surface of codimension $4$ immersed in $\mathbb{E}^6$. The formulas of Gauss are then given by
\begin{gather}
    D_{\partial_u}\partial_u = -\frac{c_1c_2}{c_1 + c_2}\overrightarrow{x},\label{formGaussgeod1}\\
    D_{\partial_u}\partial_v = D_{\partial_v}\partial_u = -\sqrt{\frac{c_1c_2}{c_1 + c_2}}\tan(\sqrt{\frac{c_1c_2}{c_1 + c_2}}u)\partial_v,\label{formGaussgeod2}\\
    D_{\partial_v}\partial_v = \sqrt{\frac{c_1c_2}{c_1 + c_2}}\cos(\sqrt{\frac{c_1c_2}{c_1 + c_2}}u)\sin(\sqrt{\frac{c_1c_2}{c_1 + c_2}}u)\partial_u - \frac{c_1c_2}{c_1 + c_2}\cos^2(\sqrt{\frac{c_1c_2}{c_1 + c_2}}u)\overrightarrow{x}\label{formGaussgeod3},
\end{gather}
where $\overrightarrow{x}$ is the position vector of $M^2$ in $\mathbb{E}^6$. Solving equations (\ref{formGaussgeod1}) and (\ref{formGaussgeod2}) we find that $\psi$ is locally given by
\begin{multline*}
 (\cos(\sqrt{\frac{c_1c_2}{c_1 + c_2}}u)\widetilde{f}_1(v) + \sin(\sqrt{\frac{c_1c_2}{c_1 + c_2}}u)\widetilde{g}_1,\dots,\cos(\sqrt{\frac{c_1c_2}{c_1 + c_2}}u)\widetilde{f}_3(v) + \sin(\sqrt{\frac{c_1c_2}{c_1 + c_2}}u)\widetilde{g}_3, \\ \cos(\sqrt{\frac{c_1c_2}{c_1 + c_2}}u)\overline{f}_1(v) + \sin(\sqrt{\frac{c_1c_2}{c_1 + c_2}}u)\overline{g}_1,\dots,\cos(\sqrt{\frac{c_1c_2}{c_1 + c_2}}u)\overline{f}_3(v) + \sin(\sqrt{\frac{c_1c_2}{c_1 + c_2}}u)\overline{g}_3),
\end{multline*}
where $\widetilde{g} = (\widetilde{g}_1,\widetilde{g}_2,\widetilde{g}_3)$ and $\overline{g} = (\overline{g}_1,\overline{g}_2,\overline{g}_3)$ are constant vectors in $\mathbb{R}^3$. Moreover we have the following conditions
\begin{gather*}
    g(\psi_u,\psi_u) = 1,\; g(\psi_u,\psi_v)=0,\;g(\psi_v,\psi_v) = \cos^2(\sqrt{\frac{c_1c_2}{c_1 + c_2}}u),\\
    g(\psi_u,h\partial_u)= 0,\;g(\psi_u,h\partial_v) = 0,\;g(h\partial_u,h\partial_u) = \frac{4c_1c_2}{(c_1 + c_2)^2},\;g(h\partial_u,h\partial_v) = 0,\\
     g(\psi_v,h\partial_v)= 0,\;g(\psi_v,h\partial_u) = 0,\;g(h\partial_v,h\partial_v) = \frac{4c_1c_2}{(c_1 + c_2)^2}\cos^2(\sqrt{\frac{c_1c_2}{c_1 + c_2}}u),\\
    g(\psi_u,\widetilde{\xi})=0,\;g(\psi_v,\widetilde{\xi})=0,\;g(\widetilde{\xi},\widetilde{\xi})=\frac{1}{c_1},\\
    g(\psi_u,\overline{\xi})=0,\;g(\psi_v,\overline{\xi})=0,\;g(\overline{\xi},\overline{\xi})=\frac{1}{c_2},\\
    g(h\partial_u,\widetilde{\xi}) = g(h\partial_v,\widetilde{\xi}) = g(h\partial_u,\overline{\xi}) = g(h\partial_v,\overline{\xi}) = 0,
\end{gather*}
in which $\widetilde{\xi} = (\psi_1,\psi_2,\psi_3,0,0,0)$ and $\overline{\xi} = (0,0,0,\psi_4,\psi_5,\psi_6)$. This conditions are equivalent to
\begin{gather*}
\sum_{i = 1}^3\widetilde{f}_i^2 = \sum_{i = 1}^3\widetilde{g}_i^2 = \frac{1}{c_1},\\
\sum_{j = 1}^3\overline{f}_j^2 = \sum_{j = 1}^3\overline{g}_j^2 = \frac{1}{c_2},\\
\sum_{i = 1}^3\widetilde{f}_i\widetilde{g}_i = \sum_{i = 1}^3\widetilde{f}'_i\widetilde{g}_i = 0,\\
\sum_{j = 1}^3\overline{f}_j\overline{g}_j =  \sum_{j = 1}^3\overline{f}'_j\overline{g}_j = 0,\\
\sum_{i = 1}^3(\widetilde{f}'_i)^2= \frac{c_2}{c_1 + c_2},\\
\sum_{j = 1}^3(\overline{f}'_i)^2= \frac{c_1}{c_1 + c_2}.
\end{gather*}
From the above equations we can conclude that $\widetilde{f} = (\widetilde{f}_1,\widetilde{f}_2,\widetilde{f}_3)$ and $\overline{f} = (\overline{f}_1,\overline{f}_2,\overline{f}_3)$ are curves in $M^2(c_1)$ and $M^2(c_2)$, respectively. Moreover we see that $\widetilde{f}$ and $\overline{f}$ are curves of speed
$\sqrt{\frac{c_2}{c_1 + c_2}}$ and $\sqrt{\frac{c_1}{c_1 + c_2}}$, respectively. The constant vector $\widetilde{g}$ is perpendicular to the vectors $\widetilde{f}$ and $\widetilde{f}'$ and the constant vector $\overline{g}$ is perpendicular to the vectors $\overline{f}$ and $\overline{f}'$. Hence we obtain that $\widetilde{g} = \pm\sqrt{\frac{c_1 + c_2}{c_2}}\widetilde{f}\times \widetilde{f}'$ and $\overline{g} = \pm\sqrt{\frac{c_1 + c_2}{c_1}}\overline{f}\times \overline{f}'$. Since $\widetilde{g}$ and $\overline{g}$ are constant vectors we obtain that the curves $\widetilde{f}$ and $\overline{f}$ are circles of radius $\frac{1}{\sqrt{c_1}}$ and $\frac{1}{\sqrt{c_2}}$, respectively. We obtain the following proposition.

\begin{Pro}\label{pro:geod}
Let $\psi:M^2\rightarrow M^2(c_1)\times M^2(c_2)$ be a totally geodesic surface with $\lambda_1 = \lambda_2 = \frac{c_2 - c_1}{c_1 + c_2}$, then $\psi$ is locally congruent to
\begin{multline}\label{geod1}
(\cos(\sqrt{\frac{c_1c_2}{c_1 + c_2}}u)\widetilde{f}(v) + \sin(\sqrt{\frac{c_1c_2}{c_1 + c_2}}u)\sqrt{\frac{c_1 + c_2}{c_2}}\widetilde{f}(v)\times\widetilde{f}'(v),\\
\cos(\sqrt{\frac{c_1c_2}{c_1 + c_2}}u)\overline{f}(v) + \sin(\sqrt{\frac{c_1c_2}{c_1 + c_2}}u)\sqrt{\frac{c_1 + c_2}{c_1}}\overline{f}(v)\times\overline{f}'(v)),
\end{multline}
where $\widetilde{f}$ and $\overline{f}$ are geodesic circles in $M^2(c_1)$ and $M^2(c_2)$, respectively, of constant speed $\sqrt{\frac{c_2}{c_1 + c_2}}$ and $\sqrt{\frac{c_1}{c_1 + c_2}}$ if $c_1,c_2 > 0$ or to
 \begin{multline}\label{geod2}
(\cosh(\sqrt{-\frac{c_1c_2}{c_1 + c_2}}u)\widetilde{f}(v) + \sinh(\sqrt{-\frac{c_1c_2}{c_1 + c_2}}u)\sqrt{\frac{c_1 + c_2}{c_2}}\widetilde{f}(v)\boxtimes\widetilde{f}'(v),\\
\cosh(\sqrt{-\frac{c_1c_2}{c_1 + c_2}}u)\overline{f}(v) + \sinh(\sqrt{-\frac{c_1c_2}{c_1 + c_2}}u)\sqrt{\frac{c_1 + c_2}{c_1}}\overline{f}(v)\boxtimes\overline{f}'(v)),
\end{multline}
where $\widetilde{f}$ and $\overline{f}$ are geodesic curves in $M^2(c_1)$ and $M^2(c_2)$, respectively, of constant speed $\sqrt{\frac{c_2}{c_1 + c_2}}$ and $\sqrt{\frac{c_1}{c_1 + c_2}}$ if $c_1,c_2 < 0$.
\end{Pro}
\subsection{$f$ is proportional to the identity}
 Suppose now that $f = \lambda I$, and that $\lambda = \cos(2\theta)$ is a constant. Using equations (\ref{symf}) and (\ref{trans}) we see that $\widetilde{g}(hX,hY) = \sin^2(2\theta)g(X,Y)$  for every $X,Y \in TM^2$. Moreover from equation (\ref{parf}) we immediately deduce that $S_{hX}Y + s(\sigma(X,Y)) = 0$. Suppose first that $\sin{2\theta} = 0$ and hence we obtain that $\theta = 0$ or $\theta = \frac{\pi}{2}$. In the first case this means that the tangent vector fields along $M^2$ are eigenvectors of $F$ with eigenvalue $1$ and that the normal vector fields along $M^2$ are eigenvectors of $F$ with eigenvalue $-1$. It can be shown then that $M^2$ is an open part of $M^2(c_1) \times \{p_2\}$. Analogously we obtain that $M^2$ is an open part of $\{p_1\} \times M^2(c_2)$ if $\theta = \frac{\pi}{2}$.

Let $\theta$ be now a constant in $(0,\frac{\pi}{2})$. Using the fact that $S_{hX}Y + s(\sigma(X,Y)) = 0$ for every $X,Y \in TM^2$, we deduce that $v$ is an eigenvector of $S_{hv}$ with eigenvalue $0$ for every $v \in T_pM^2$, i.e. $S_{hv}v = 0$. Take now an arbitrary orthonormal basis $\{e_1,e_2\} \subset T_pM^2$. Consider the shape operators $S_{he_1}$ and $S_{he_2}$ associated to $he_1$ and $he_2$, respectively. We have then that $S_{he_1}e_2 = \mu_1e_2$ and $S_{he_2}e_1 = \mu_2e_1$. Moreover we have that $0 = S_{h(e_1 + e_2)}(e_1 + e_2) = \mu_1e_2 + \mu_2e_1$ and hence we have that $\mu_1 = \mu_2 = 0$. We conclude that $M^2$ is a totally geodesic surface in $M^2(c_1)\times M^2(c_2)$, because $\{he_1,he_2\}$ is an orthogonal basis of $T^{\perp}M^2$ and $S_{he_1} = S_{he_2} = 0$. Using the equation of Codazzi, we obtain that
\[(c_1\cos^2(\theta) - c_2\sin^2(\theta))\sin(\theta)\cos(\theta) = 0.\]
Since $\theta \in (0,\frac{\pi}{2})$ and $c_1$ and $c_2$ are not both $0$, we obtain that $c_1c_2 >0$ and $\tan^2(\theta) = \frac{c_1}{c_2}$. Hence we have that $\cos(2\theta) = \frac{c_2 - c_1}{c_1 + c_2}$. Moreover we have that the Gaussian curvature of the surface equals $\frac{c_1c_2}{c_1 + c_2}$. We summarize the previous in the following proposition.

\begin{Pro}
Let $M^2$ be a surface immersed in $M^2(c_1)\times M^2(c_2)$. Suppose $f = \lambda I$ with $\lambda \in [-1,1]$ and $\lambda$ is a constant. Then $M^2$ is a totally geodesic surface in $M^2(c_1)\times M^2(c_2)$. Moreover we have that the $\lambda$ equals $1,-1$ or $\frac{c_2  - c_1}{c_1 + c_2}$ with $c_1c_2 > 0$ and that the Gaussian curvature of $M^2$ equals $c_1,\; c_2$ and $\frac{c_1c_2}{c_1 + c_2}$. Hence the surface is an open part of $M^2(c_1)\times\{p_2\}$, $\{p_1\} \times M^2(c_2)$ or is locally given by  (\ref{geod1}) or (\ref{geod2}).
\end{Pro}
\subsection{$f$ is not proportional to the identity}
We first consider the  trivial case $\theta_2 = 0$ and $\theta_1 = \frac{\pi}{2}$. One can then easily prove that $M^2$ is an open part of a Riemannian product of a curve in $M^2(c_1)$ and a curve of $M^2(c_2)$.

Suppose now that $\theta_1 =\frac{\pi}{2}$ and $\theta_2$ is a constant in $(0,\frac{\pi}{2})$. Denote in the following $\theta_2$ by $\theta$.  Consider an adapted orthonormal frame $\{e_1,e_2,\xi_1,\xi_2\}$ such that $fe_1 = -e_1,\,fe_2= \cos(2\theta)e_2,\,t\xi_1 = \xi_1$ and $t\xi_2 = -\cos(2\theta)\xi_2$. Using equations (\ref{symf}), (\ref{trans}) and (\ref{orth}), we see that $he_2 = \pm \sin(2\theta)\xi_2$. We may suppose that $he_2 = \sin(2\theta)\xi_2$. Moreover we can deduce from equations (\ref{symf}) and (\ref{trans}) that $he_1 = 0$. Using equations (\ref{parf}), (\ref{parh}) and (\ref{part}), we obtain that
\begin{gather}
2\cos^2(\theta)\nabla_Xe_2 = \sin(2\theta)S_{\xi_2}X + s(\sigma(X,e_2)),\label{covf2}\\
-\sin(2\theta)g(\nabla_Xe_1,e_2)\xi_2 = t(\sigma(X,e_1)) + \sigma(X,e_1),\label{covh2}\\
2\cos^2(\theta)\nabla_{X}^{\perp}\xi_2 = \sin(2\theta)\sigma(e_2,X) + h(S_{\xi_2}X).\label{covt2}
\end{gather}
From equations (\ref{covf2}) and (\ref{covh2}) we deduce that $g(S_{\xi_2}X,e_2) = g(S_{\xi_1}X,e_1) = 0$ for every $X \in TM^2$. Hence we obtain, using equations (\ref{covf2}) and (\ref{covt2}) and the fact that $g(S_{\xi_2}X,e_2) = g(S_{\xi_1}X,e_1) = 0$, that
\begin{gather*}
g(\nabla_Xe_1,e_2) = -\tan(\theta)\mu_2g(X,e_1),\\
\widetilde{g}(\nabla_X\xi_1,\xi_2) = -\tan(\theta)\mu_1g(X,e_2),
\end{gather*}
where $\mu_1$ is the eigenvalue of $S_{\xi_1}$ and $\mu_2$ is the eigenvalue of $S_{\xi_2}$. Since we know the shape-operators $S_{\xi_1}$ and $S_{\xi_2}$ and the symmetric operator $f$, we can find the Gaussian curvature $K$ of $M^2$. From the equation of Gauss we find that $K = c_2\sin^2(\theta)$. From the equation of Ricci we obtain also easily that $K^{\perp} = |\widetilde{g}(R^{\perp}(e_1,e_2)\xi_2,\xi_1)| = 0$. We summarize the previous in the following proposition.

\begin{Pro}\label{cond1}
Let $M^2$ be a constant angle surface immersed in $M^2(c_1) \times M^2(c_2)$. Suppose that $\lambda_1 = -1$ and $\lambda_2$ is a constant in $(-1,1)$. Then we can find an adapted frame $\{e_1,e_2,\xi_1,\xi_2\}$ such that $fe_1 = -e_1,fe_2 = \cos(2\theta)e_2,\,t\xi_1 = \xi_1$ and $t\xi_2 = -\cos(2\theta)\xi_2$, where $\cos(2\theta) = \lambda_2$ and such that the shape operators $S_{\xi_1}$ and $S_{\xi_2}$ take the following form with respect to the orthonormal frame $\{e_1,e_2\}:$
\begin{equation}\label{shape}
S_{\xi_1} = \left(
  \begin{array}{cc}
    0 & 0 \\
    0 & \mu_1 \\
  \end{array}
\right),
\qquad
S_{\xi_2} = \left(
  \begin{array}{cc}
    \mu_2 & 0 \\
    0 & 0 \\
  \end{array}
\right),
\end{equation}
for some functions $\mu_1$ and $\mu_2$ on $M^2$. Moreover the Levi-Civita connection $\nabla$ of $M^2$ and the normal connection $\nabla^{\perp}$ of $M^2$ in $M^2(c_1) \times M^2(c_2)$ are given by
\begin{gather}
g(\nabla_Xe_1,e_2) = -\tan(\theta)\mu_2g(X,e_1),\label{Levi1}\\
\widetilde{g}(\nabla^{\perp}_X\xi_1,\xi_2) = -\tan(\theta)\mu_1 g(X,e_2)\label{norLevi1}.
\end{gather}
The Gaussian curvature $K$ is given by
\begin{equation*}
K = c_2\sin^2(\theta),
\end{equation*}
and the normal curvature $K^{\perp}$ is equal to $0$.
\end{Pro}
We obtain a similar proposition if $\lambda_1$ is a constant in $(-1,1)$ and $\lambda_2 = 1$.
\begin{Pro}\label{cond2}
Let $M^2$ be a constant angle surface immersed in $M^2(c_1) \times M^2(c_2)$. Suppose that $\lambda_1$ is a constant in $(-1,1)$ and $\lambda_2 = 1$. Then we can find an adapted frame $\{e_1,e_2,\xi_1,\xi_2\}$ such that $fe_1 = \cos(2\theta)e_1,fe_2 = e_2,\,t\xi_1 = -\cos(2\theta)\xi_1$ and $t\xi_2 = -\xi_2$, where $\cos(2\theta) = \lambda_1$and such that the shape operators $S_{\xi_1}$ and $S_{\xi_2}$ take the following form with respect to the orthonormal frame $\{e_1,e_2\}:$
\begin{equation*}
S_{\xi_1} = \left(
  \begin{array}{cc}
    0 & 0 \\
    0 & \mu_1 \\
  \end{array}
\right),
\qquad
S_{\xi_2} = \left(
  \begin{array}{cc}
    \mu_2 & 0 \\
    0 & 0 \\
  \end{array}
\right),
\end{equation*}
for some functions $\mu_1$ and $\mu_2$ on $M^2$. Moreover the Levi-Civita connection $\nabla$ of $M^2$ and the normal connection $\nabla^{\perp}$ of $M^2$ in $M^2(c_1) \times M^2(c_2)$ are given by
\begin{gather}
g(\nabla_Xe_1,e_2) = -\cot(\theta)\mu_1g(X,e_2),\label{Levi2}\\
\widetilde{g}(\nabla^{\perp}_X\xi_1,\xi_2) = -\cot(\theta)\mu_2 g(X,e_1)\label{norLevi2}.
\end{gather}
The Gaussian curvature $K$ is given by
\begin{equation}
K = c_1\cos^2(\theta),
\end{equation}
and the normal curvature $K^{\perp}$ is equal to $0$.
\end{Pro}
Finally we consider the case for which $\lambda_1 = \cos(2\theta_1)$ and $\lambda_2 = \cos(2\theta_2)$ are constant, $\lambda_2 - \lambda_1 > 0$ and $\lambda_1,\lambda_2 \in (-1,1)$. Let $\{e_1,e_2,\xi_1,\xi_2\}$ be an adapted frame such that $fe_i = \cos(2\theta_i)e_i$ and $t\xi_i = -\cos(2\theta_i)\xi_i$ for $i=1,2$. Using equations (\ref{parf}) and (\ref{part}) and by similar reasoning as before, we obtain the next proposition.
\begin{Pro}
Let $M^2$ be a surface immersed in $M^2(c_1)\times M^2(c_2)$. Suppose that $\lambda_1$ and $\lambda_2$ are constants in $(-1,1)$ and $\lambda_2 - \lambda_1 >0$. Then we can find an adapted orthonormal frame $\{e_1,e_2,\xi_1,\xi_2\}$ such that $fe_i = \cos(2\theta_i) e_i$ and $t\xi_i = -\cos(2\theta_i)\xi_i$, where $\cos(2\theta_i) = \lambda_i$ for $i = 1,2$ and such that the shape operators $S_{\xi_1}$ and $S_{\xi_2}$ take the following form with respect to the orthonormal frame $\{e_1,e_2\}$:
\begin{equation*}
S_{\xi_1} = \left(
  \begin{array}{cc}
    0 & 0 \\
    0 & \mu_1 \\
  \end{array}
\right),
\qquad
S_{\xi_2} = \left(
  \begin{array}{cc}
    \mu_2 & 0 \\
    0 & 0 \\
  \end{array}
\right),
\end{equation*}
for some functions $\mu_1$ and $\mu_2$ on $M^2$. Moreover the Levi-Civita connection $\nabla$ of $M^2$ and the normal connection $\nabla^{\perp}$ of $M^2$ in $M^2(c_1)\times M^2(c_2)$ are given by:
\begin{gather}
g(\nabla_Xe_1,e_2) = \frac{\cos(\theta_1)\sin(\theta_1)\mu_1g(X,e_2) + \cos(\theta_2)\sin(\theta_2)\mu_2g(X,e_1)}{\cos^2(\theta_1) - \cos^2(\theta_2)},\label{Levi4}\\
\widetilde{g}(\nabla^{\perp}_X\xi_1,\xi_2) = \frac{\cos(\theta_1)\sin(\theta_1)\mu_2g(X,e_1) + \cos(\theta_2)\sin(\theta_2)\mu_1g(X,e_2)}{\cos^2(\theta_1) - \cos^2(\theta_2)}.\label{norLevi4}
\end{gather}
The Gaussian curvature $K$ is given by
\begin{equation}\label{Gauss3}
K = c_1\cos^2(\theta_1)\cos^2(\theta_2) + c_2\sin^2(\theta_1)\sin^2(\theta_2),
\end{equation}
and  the normal curvature $K^{\perp}$ equals $|\frac{c_1 + c_2}{4}|\sin(2\theta_1)\sin(2\theta_2)$.
\end{Pro}

\section{Existence results}

We will need the following existence results in the next section.
\begin{Pro}\label{Exis1}
Let $c_1,c_2 \in \mathbb{R}$, not both $0$, $\theta_1,\theta_2 \in (0,\frac{\pi}{2})$ with $\theta_1 > \theta_2$. Define constants $a_1,a_2,A_1$ and $A_2$ by
\[a_1 = \frac{\sin(2\theta_1)}{\cos(2\theta_1) - \cos(2\theta_2)},\qquad a_2 = \frac{\sin(2\theta_2)}{\cos(2\theta_2) - \cos(2\theta_1)},\]
\[A_1 = (\cos^2(\theta_2) - \cos^2(\theta_1))(c_1\cos^2(\theta_2) - c_2\sin^2(\theta_2)),\]
and
\[A_2 = (\cos^2(\theta_1) - \cos^2(\theta_2))(c_1\cos^2(\theta_1) - c_2\sin^2(\theta_1)).\]
Let $\mu_1 = \mu_1(u,v)$ and $\mu_2 = \mu_2(u,v)$ be real-valued functions defined on a simply connected open subset of $\mathbb{R}^2$ which satisfy
\begin{equation}\label{exiscon1}
-\frac{a_1}{\sqrt{\mu_2^2 + A_2}} = \frac{(\mu_1)_u}{\mu_1^2 +A_1},\qquad -\frac{a_2}{\sqrt{\mu_1^2 + A_1}} =  \frac{(\mu_2)_v}{\mu_2^2 +A_2}.
\end{equation}
Then the Riemannian manifold $M^2 = (U,g)$ with the Riemannian metric $g = \frac{du^2}{\mu_2^2 + A_2} + \frac{dv^2}{\mu_1^2 + A_1}$ is a surface of constant curvature $c_1\cos^2(\theta_1)\cos^2(\theta_2) + c_2\sin^2(\theta_1)\sin^2(\theta_2)$. Define now on the vector bundle $TU$ a second metric $\widetilde{g}$ by $\sin^2(2\theta_1)\frac{du^2}{\mu_2^2 + A_2} + \sin^2(2\theta_2)\frac{dv^2}{\mu_1^2 + A_1}$ and denote this Riemannian vector bundle by $T^{\perp}M^2$. Let $f: TM^2 \rightarrow TM^2$, $t:T^{\perp}M^2 \rightarrow T^{\perp}M^2$, and $h :TM^2 \rightarrow T^{\perp}M^2$ be, respectively, $(1,1)$-tensors over $M^2$ defined by
\begin{equation*}
\left(
  \begin{array}{cc}
    \cos(2\theta_1) & 0 \\
    0 & \cos(2\theta_2) \\
  \end{array}
\right),\qquad
\left(
  \begin{array}{cc}
    -\cos(2\theta_1) & 0 \\
    0 & -\cos(2\theta_2) \\
  \end{array}
\right),\qquad
\left(
  \begin{array}{cc}
    1 & 0 \\
    0 & 1 \\
  \end{array}
\right),
\end{equation*}
with respect to $\{\partial_u,\partial_v\}$ and $\{\widetilde{\partial}_u,\widetilde{\partial}_v\}$, where $\widetilde{\partial}_u,\widetilde{\partial}_v \in T^{\perp}M^2$ and dual to the forms $du$ and $dv$. Finally define a symmetric $(1,2)$ tensor $\sigma$ with values in $T^{\perp}M^2$ and a connection $\nabla^{\perp}$ on $T^{\perp}M^2$ compatible with the metric by
\begin{equation}\label{secfundexis1}
\begin{split}
\sigma(\partial_u,\partial_u) = \frac{\mu_2\sqrt{\mu_1^2 + A_1}}{\sin(2\theta_2)(\mu_2^2 + A_2)}h\partial_v,\:\sigma(\partial_u,\partial_v) = 0,\\
\sigma(\partial_v,\partial_v) = \frac{\mu_1\sqrt{\mu_2^2 + A_2}}{\sin(2\theta_1)(\mu_1^2) + A_1}h\partial_u,
\end{split}
\end{equation}
\begin{equation}\label{norconexis1}
\begin{split}
\nabla^{\perp}_{\partial_u}h\partial_u = -\frac{\mu_2(\mu_2)_u}{\mu_2^2  + A_2}h\partial_u + \frac{a_1\sin(2\theta_1)\mu_2\sqrt{\mu_1^2 + A_1}}{\sin(2\theta_2)(\mu_2^2 + A_2)}h\partial_v,\notag\\
\nabla^{\perp}_{\partial_u}h\partial_v = \nabla^{\perp}_{\partial_v}h\partial_u = h(\nabla_{\partial_u}\partial_v) = h(\nabla_{\partial_v}\partial_u),\notag\\
\nabla^{\perp}_{\partial_v}h\partial_v =  \frac{a_2\sin(2\theta_2)\mu_1\sqrt{\mu_2^2 + A_2}}{\sin(2\theta_1)(\mu_1^2 + A_1)}h\partial_u - \frac{\mu_1(\mu_1)_v}{\mu_1^2 + A_1}h\partial_v,\notag
\end{split}
\end{equation}
where $\nabla$ is the Levi-Civita connection of $M^2$. Then $(M^2,g,T^{\perp}M^2,\widetilde{g},\sigma,\nabla^{\perp}f,h,t)$ satisfies the compatibility equations of $M^2(c_1)\times M^2(c_2)$ and hence there exists an isometric immersion of $M^2$ in $M^2(c_1)\times M^{2}(c_2)$ such that this surface is a constant angle surface and is unique up to isometries of $M^2(c_1)\times M^2(c_2)$.
\end{Pro}

\begin{proof}
From (\ref{exiscon1}) and a direct computation we know that the Riemannian metric $g$ has constant curvature $c_1\cos^2(\theta_1)\cos^2(\theta_2) + c_2\sin^2(\theta_1)\sin^2(\theta_2)$ and the Levi-Civita connection satisfies
\begin{equation*}
\begin{split}
\nabla_{\partial_u}\partial_u = -\frac{\mu_2(\mu_2)_u}{\mu_2^2 + A_2} \partial_u - \frac{a_2\mu_2\sqrt{\mu_1^2 + A_1}}{\mu_2^2 + A_2}\partial_v,\\
\nabla_{\partial_u}\partial_v = \nabla_{\partial_v}\partial_u  = \frac{a_2\mu_2}{\sqrt{\mu_1^2 + A_1}}\partial_u + \frac{a_1\mu_1}{\sqrt{\mu_2^2 + A_2}}\partial_v,\\
\nabla_{\partial_v}\partial_v = -\frac{a_1\mu_1\sqrt{\mu_2^2 + A_2}}{\mu_1^2 + A_1} - \frac{\mu_1(\mu_1)_v}{\mu_1^2 + A_1}\partial_v.
\end{split}
\end{equation*}
 We have already defined a second metric $\widetilde{g}$ on the vector bundle $TU$, and denoted this Riemannian vector bundle by $T^{\perp}M^2$, together with a connection $\nabla^{\perp}$ that is compatible with this metric. Let $f,t$ and $h$ be $(1,1)$-tensors as defined above and $\sigma$ the symmetric $(1,2)$ tensor defined by (\ref{secfundexis1}). By direct straightforward computations we can see that $(M^2,g,T^{\perp}M^2,\widetilde{g},\nabla^{\perp},\sigma,f,h,t)$ satisfies the compatibility equations for $M^2(c_1) \times M^2(c_2)$. Hence there exists an isometric immersion of $M^2$ into $M^2(c_1) \times M^2(c_2)$. Moreover, we can deduce from equation (\ref{structure}) that $M^2$ is a constant angle surface in $M^2(c_1) \times M^2(c_2)$. We can also conclude from Theorem $2$ that this immersion with the given second fundamental form and normal connection is unique up to rigid motions of $M^2(c_1) \times M^2(c_2)$.
\end{proof}
The next two propositions can be proven analogously as the previous one.
\begin{Pro}\label{Exis2}
Let $c_1,c_2 \in \mathbb{R}$, not both $0$, $\theta_1,\theta_2 \in (0,\frac{\pi}{2})$ with $\theta_1 > \theta_2$. Define constants $a_1,a_2,A_1$ and $A_2$ as in Proposition $\ref{Exis1}$. Moreover we suppose that $A_1 < 0$. Let $\mu  = \mu(u,v)$ and $G = G(u,v)$ be real-valued functions which satisfy
\begin{equation}\label{exiscon2}
-a_2\sqrt{G} = \frac{(\mu)_v}{\mu^2 + A_2},\qquad a_1\sqrt{\frac{-A_1}{\mu^2 + A_2}} = \frac{G_u}{2G}.
\end{equation}
Then the Riemannian manifold $M^2 = (U,g)$ with the Riemannian metric $g = \frac{du^2}{\mu^2 + A_2} +Gdv^2$
is a surface of constant curvature $c_1\cos^2(\theta_1)\cos^2(\theta_2) + c_2\sin^2(\theta_1)\sin^2(\theta_2)$. Define now on the vector bundle $TU$ a second metric $\widetilde{g}$ by $\sin^2(2\theta_1)\frac{du^2}{\mu_2^2 + A_2} + \sin^2(2\theta_2)Gdv^2$ and denote this Riemannian vector bundle by $T^{\perp}M^2$. Let $f: TM^2 \rightarrow TM^2$, $t:T^{\perp}M^2 \rightarrow T^{\perp}M^2$, and $h :TM^2 \rightarrow T^{\perp}M^2$ be $(1,1)$-tensors over $M^2$ as defined above in Proposition $\ref{Exis1}$. Finally define a symmetric $(1,2)$ tensor $\sigma$ with values in $T^{\perp}M^2$ and a connection $\nabla^{\perp}$ on $T^{\perp}M^2$ compatible with the metric by
\begin{equation}\label{secfundexis2}
\begin{split}
\sigma(\partial_u,\partial_u) = \frac{\mu_2}{\sin(2\theta_2)(\mu_2^2 + A_2)\sqrt{G}}h\partial_v,\:\sigma(\partial_u,\partial_v) = 0,\\
\sigma(\partial_v,\partial_v) = \frac{G\sqrt{-A_1(\mu_2^2 + A_2)}}{\sin(2\theta_1)}h\partial_u,
\end{split}
\end{equation}
\begin{equation*}
\begin{split}
\nabla^{\perp}_{\partial_u}h\partial_u = -\frac{\mu_2(\mu_2)_u}{\mu_2^2  + A_2}h\partial_u + \frac{a_1\sin(2\theta_1)\mu_2}{\sin(2\theta_2)(\mu_2^2 + A_2)\sqrt{G}}h\partial_v,\\
\nabla^{\perp}_{\partial_u}h\partial_v = \nabla^{\perp}_{\partial_v}h\partial_u = h(\nabla_{\partial_u}\partial_v) = h(\nabla_{\partial_v}\partial_u),\\
\nabla^{\perp}_{\partial_v}h\partial_v =  \frac{a_2\sin(2\theta_2)\mu_1\sqrt{-A_1(\mu_2^2 + A_2)}G}{\sin(2\theta_1)}h\partial_u  + \frac{G_v}{2G}h\partial_v,
\end{split}
\end{equation*}
where $\nabla$ is the Levi-Civita connection of $M^2$. Then $(M^2,g,T^{\perp}M^2,\widetilde{g},\sigma,\nabla^{\perp},f,h,t)$ satisfies the compatibility equations of $M^2(c_1)\times M^2(c_2)$ and hence there exists an isometric immersion of $M^2$ in $M^2(c_1)\times M^{2}(c_2)$ such that this surface is a constant angle surface and is unique up to isometries of $M^2(c_1)\times M^2(c_2)$.
\end{Pro}

\begin{Pro}\label{Exis3}
Let $c_1,c_2 \in \mathbb{R}$, not both $0$, $\theta_1,\theta_2 \in (0,\frac{\pi}{2})$ with $\theta_1 > \theta_2$. Define constants $a_1,a_2,A_1$ and $A_2$ as in Proposition $\ref{Exis1}$. Moreover we suppose that $A_1, A_2 < 0$. Let $E  = E(u,v)$ and $G = G(u,v)$ be positive real-valued functions which satisfy
\begin{equation}\label{exiscon3}
a_2\sqrt{-A_2G} = \frac{E_v}{2E},\qquad a_1\sqrt{-A_1E} = \frac{G_u}{2G}.
\end{equation}
Then the Riemannian manifold $M^2 = (U,g)$ with the Riemannian metric $g = Edu^2 +Gdv^2$
is a surface of constant curvature $c_1\cos^2(\theta_1)\cos^2(\theta_2) + c_2\sin^2(\theta_1)\sin^2(\theta_2)$. Define now on the vector bundle $TU$ a second metric $\widetilde{g}$ by $\sin^2(2\theta_1)Edu^2 + \sin^2(2\theta_2)Gdv^2$ and denote this Riemannian vector bundle by $T^{\perp}M^2$. Let $f: TM^2 \rightarrow TM^2$, $t:T^{\perp}M^2 \rightarrow T^{\perp}M^2$, and $h :TM^2 \rightarrow T^{\perp}M^2$ be $(1,1)$-tensors over $M^2$ as defined in Proposition $\ref{Exis1}$. Finally define a symmetric $(1,2)$ tensor $\sigma$ with values in $T^{\perp}M^2$ and a connection $\nabla^{\perp}$ on $T^{\perp}M^2$ compatible with the metric by
\begin{equation*}
\begin{split}
\sigma(\partial_u,\partial_u) = \frac{\sqrt{-A_2}E}{\sin(2\theta_2)\sqrt{G}} h\partial_v,\:\sigma(\partial_u,\partial_v) = 0,\\
\sigma(\partial_v,\partial_v) = \frac{\sqrt{-A_1}G}{\sin(2\theta_1)\sqrt{E}}h\partial_u,
\end{split}
\end{equation*}
\begin{equation*}
\begin{split}
\nabla^{\perp}_{\partial_u}h\partial_u = \frac{E_u}{2E}h\partial_u  + \frac{a_1\sin(2\theta_1)\sqrt{-A_2}E}{\sqrt{G}}h\partial_v,\\
\nabla^{\perp}_{\partial_u}h\partial_v = \nabla^{\perp}_{\partial_v}h\partial_u = h(\nabla_{\partial_u}\partial_v) = h(\nabla_{\partial_v}\partial_u),\\
\nabla^{\perp}_{\partial_v}h\partial_v = \frac{a_2\sin(2\theta_2)\sqrt{-A_1}G}{\sqrt{E}}h\partial_u + \frac{G_v}{2G}h\partial_v,
\end{split}
\end{equation*}
where $\nabla$ is the Levi-Civita connection of $M^2$. Then $(M^2,g,T^{\perp}M^2,\widetilde{g},\sigma,\nabla^{\perp},f,h,t)$ satisfies the compatibility equations of $M^2(c_1)\times M^2(c_2)$ and hence there exists an isometric immersion of $M^2$ in $M^2(c_1)\times M^{2}(c_2)$ such that this surface is a constant angle surface and is unique up to isometries of $M^2(c_1)\times M^2(c_2)$.
\end{Pro}

\subsection{Sine-Gordon and Sinh-Gordon equations}

We would like to make some remarks on the equations (\ref{exiscon1}). Suppose that $A_1,A_2 > 0$. If we define functions $\widetilde{\theta}_i$ such that $\mu_i = \sqrt{A_i}\cot(\widetilde{\theta_i})$, then the equations (\ref{exiscon1}) are equivalent to
\begin{equation}\label{Backtrans}
\begin{split}
-a_1\sqrt{\frac{A_1}{A_2}}\sin(\widetilde{\theta}_2) = (\widetilde{\theta}_1)_u,\\
-a_2\sqrt{\frac{A_2}{A_1}}\sin(\widetilde{\theta}_1) = (\widetilde{\theta}_2)_v.
\end{split}
\end{equation}
Differentiating the first equation with respect to $v$ and second equation with respect to $u$ gives us
\begin{gather}(\widetilde{\theta_1})_{uv} = -a_1\sqrt{\frac{A_1}{A_2}}\cos(\widetilde{\theta}_2)(\widetilde{\theta}_2)_v = a_1a_2\cos(\widetilde{\theta}_2)\sin(\widetilde{\theta}_1),\label{tussen}\\
(\widetilde{\theta_2})_{vu} = -a_2\sqrt{\frac{A_2}{A_1}}\cos(\widetilde{\theta}_1)(\widetilde{\theta}_1)_u = a_1a_2\cos(\widetilde{\theta}_1)\sin(\widetilde{\theta}_2)\label{tussensec}.
\end{gather}
The operations $(\ref{tussen}) + (\ref{tussensec})$ and $(\ref{tussen}) - (\ref{tussensec})$ yield
\begin{gather*}
(\widetilde{\theta_1} + \widetilde{\theta}_2)_{uv} = a_1a_2\sin(\widetilde{\theta_1} + \widetilde{\theta}_2),\\
(\widetilde{\theta_1} - \widetilde{\theta}_2)_{uv} = a_1a_2\sin(\widetilde{\theta_1} - \widetilde{\theta}_2).
\end{gather*}
Hence we have found a correspondence between some constant angle surfaces in $M^2(c_1)\times M^2(c_2)$ and the Sine-Gordon equation. We would like to remark that the equations (\ref{Backtrans}) are the B\"acklund transformations for this Sine-Gordon equation. So we obtain a big range of surfaces with constant angle in $M^2(c_1)\times M^2(c_2)$.

With similar reasoning we find a correspondence with the Sinh-Gordon equation and some constant angle surfaces in $M^2(c_1)\times M^2(c_2)$ if $A_1,A_2 < 0$. We define functions $\widetilde{\theta}_i:U \rightarrow (0,\infty)$, such that $\mu_i = \sqrt{-A_i}\coth(\widetilde{\theta}_i)$, then the equations (\ref{exiscon1}) are equivalent to
\begin{equation*}
\begin{split}
-a_1\sqrt{\frac{A_1}{A_2}}\sinh(\widetilde{\theta}_2) = (\widetilde{\theta}_1)_u,\\
-a_2\sqrt{\frac{A_2}{A_1}}\sinh(\widetilde{\theta}_1) = (\widetilde{\theta}_2)_v.
\end{split}
\end{equation*}
Differentiating the first equation with respect to $v$ and second equation with respect to $u$ gives us
\begin{gather}(\widetilde{\theta_1})_{uv} = -a_1\sqrt{\frac{A_1}{A_2}}\cosh(\widetilde{\theta}_2)(\widetilde{\theta}_2)_v = a_1a_2\cosh(\widetilde{\theta}_2)\sinh(\widetilde{\theta}_1),\label{tussen1}\\
(\widetilde{\theta_2})_{vu} = -a_2\sqrt{\frac{A_2}{A_1}}\cosh(\widetilde{\theta}_1)(\widetilde{\theta}_1)_u = a_1a_2\cosh(\widetilde{\theta}_1)\sinh(\widetilde{\theta}_2)\label{tussensec1}.
\end{gather}
The operations $(\ref{tussen1}) + (\ref{tussensec1})$ and $(\ref{tussen1}) - (\ref{tussensec1})$ yield
\begin{gather}
(\widetilde{\theta_1} + \widetilde{\theta}_2)_{uv} = a_1a_2\sinh(\widetilde{\theta_1} + \widetilde{\theta}_2),\\
(\widetilde{\theta_1} - \widetilde{\theta}_2)_{uv} = a_1a_2\sinh(\widetilde{\theta_1} - \widetilde{\theta}_2).
\end{gather}
Finally we suppose that $A_1 < 0$ and $A_2 > 0$. We define functions $\widetilde{\theta}_1:U \rightarrow (0,\infty)$, such that $\mu_1 = \sqrt{-A_1}\coth(\widetilde{\theta}_i)$, and  $\widetilde{\theta}_2:U \rightarrow (0,\pi)$, such that $\mu_2 = \sqrt{A_2}\cot(\widetilde{\theta}_2)$, then the equations (\ref{exiscon1}) are equivalent to
\begin{equation*}
\begin{split}
-a_1\sqrt{-\frac{A_1}{A_2}}\sin(\widetilde{\theta}_2) = (\widetilde{\theta}_1)_u,\\
-a_2\sqrt{-\frac{A_2}{A_1}}\sinh(\widetilde{\theta}_1) = (\widetilde{\theta}_2)_v.
\end{split}
\end{equation*}
Differentiating the first equation with respect to $v$ and the second equation with respect to $u$ gives us
\begin{gather}(\widetilde{\theta_1})_{uv} = -a_1\sqrt{-\frac{A_1}{A_2}}\cos(\widetilde{\theta}_2)(\widetilde{\theta}_2)_v = a_1a_2\cos(\widetilde{\theta}_2)\sinh(\widetilde{\theta}_1),\label{tussen2}\\
(\widetilde{\theta_2})_{vu} = -a_2\sqrt{-\frac{A_2}{A_1}}\cosh(\widetilde{\theta}_1)(\widetilde{\theta}_1)_u = a_1a_2\cosh(\widetilde{\theta}_1)\sin(\widetilde{\theta}_2)\label{tussensec2}.
\end{gather}
The operations $(\ref{tussen2}) + i(\ref{tussensec2})$ and $(\ref{tussen2}) - i(\ref{tussensec2})$ yield
\begin{gather*}
(\widetilde{\theta_1} + i\widetilde{\theta}_2)_{uv} = a_1a_2\sinh(\widetilde{\theta_1} + i\widetilde{\theta}_2),\\
(\widetilde{\theta_1} - i\widetilde{\theta}_2)_{uv} = a_1a_2\sinh(\widetilde{\theta_1} - i\widetilde{\theta}_2).
\end{gather*}
\section{Main Theorems}
In this final section we will classify all the constant angle surfaces in $M^2(c_1)\times M^2(c_2)$. We split the classification in several subcases. Suppose first that $\lambda_1 = -1$ and $\lambda_2 = \cos(2\theta)$ is a constant in $(-1,1)$. We will prove the following theorem.
\begin{Th}\label{Th:const1}
A surface $M^2$ isometrically immersed in $M^2(c_1)\times M^2(c_2)$ is a constant angle surface with angles $\theta$ and $\frac{\pi}{2}$ if and only if the immersion $\psi$ is locally given by
\begin{equation*}
\psi(u,v) = (\widetilde{f}(v),\cos(\sqrt{c_2}\sin(\theta)v)\bar{f}(u) + \sin(\sqrt{c_2}\sin(\theta)v)\bar{f}(u)\times\bar{f}'(u)) \hbox{ if } c_2>0,
\end{equation*}
where $\widetilde{f}$ is a curve in $M^2(c_1)$ of constant speed $\cos(\theta)$ and $\bar{f}$ is a unit speed curve in $M^2(c_2)$; by
\begin{equation*}
\psi(u,v) = (\widetilde{f}(v),\cosh(\sqrt{-c_2}\sin(\theta)v)\bar{f}(u) + \sinh(\sqrt{-c_2}\sin(\theta)v)\bar{f}(u)\boxtimes\bar{f}'(u)) \hbox{ if }c_2<0,
\end{equation*}
where $\widetilde{f}$ is a curve in $M^2(c_1)$ of constant speed $\cos(\theta)$ and $\bar{f}$ is a unit speed curve in $M^2(c_2)$,
\begin{equation}\label{geod3}
\psi(u,v) = (\widetilde{f}(v),u,\sin(\theta)v)\: or \:\psi(u,v) = (\widetilde{f}(v), v\sin(\theta)\bar{f}(u) + \bar{g}(u)) \hbox{ if } c_2 = 0; or by
\end{equation}
where $\widetilde{f}$ is a curve in $M^2(c_1)$ of constant speed $\cos(\theta)$, $\bar{f}(u) = (\cos(u),\sin(u))$ and $\bar{g}'(u) = \cos(\theta)C(u)(-\sin(u),\cos(u))$, where $C$ is a function on an interval $I$.
\end{Th}
\begin{proof}
After a straight-forward computation, one can verify that the surfaces listed in the theorem are constant angle surfaces in $M^2(c_1)\times M^2(c_2)$ with angles $\theta$ and $\frac{\pi}{2}$.

Conversely, let $\psi: M^2 \rightarrow M^2(c_1)\times M^2(c_2)$ be a constant angle surface, with $\lambda_1 = -1$ and $\lambda_2 = \cos(2\theta)$. Then Proposition (\ref{cond1}) tells us that we can find an adapted orthonormal frame $\{e_1,e_2,\xi_1,\xi_2\}$ such that $fe_1 = -e_1,\:fe_2 = \lambda_2e_2,\:t\xi_1 = \xi_1$ and $t\xi_2=-\lambda_2\xi_2$ and such that the shape operators associated to $\xi_1$ and $\xi_2$ with respect to $e_1$ and $e_2$ are given by
\begin{equation*}
S_{\xi_1} = \left(
  \begin{array}{cc}
    0 & 0 \\
    0 & \mu_1 \\
  \end{array}
\right),
\qquad
S_{\xi_2} = \left(
  \begin{array}{cc}
    \mu_2 & 0 \\
    0 & 0 \\
  \end{array}
\right),
\end{equation*}
for some functions $\mu_1$ and $\mu_2$ on $M^2$. Using (\ref{Levi1}) we obtain that the Levi-Civita connection satisfies
\begin{gather}
\nabla_{e_1}e_1 = \tan(\theta)\mu_2 e_2,\\
\nabla_{e_1}e_2 = -\tan(\theta)\mu_2 e_1,\label{conc1}\\
\nabla_{e_2}e_1 = 0,\label{conc2}\\
\nabla_{e_2}e_2 = 0.
\end{gather}
From equations (\ref{conc1}) and (\ref{conc2}) and the fact that $[\partial_u,\partial_v] = 0$, we can deduce that there exist locally coordinates $(u,v)$ on $M^2$ such that $\partial_{u} = \alpha e_1$ and $\partial_{v} = e_2$ with
\begin{equation}\label{int1}
\alpha_{v} = \alpha\mu_2\tan(\theta).
\end{equation}
Hence the metric takes the form
\begin{equation*}
ds^2 = \alpha^2 du^2 + dv^2
\end{equation*}
and the Levi-Civita connection is given by
\begin{gather*}
\nabla_{\partial_u} \partial_u= \frac{\alpha_u}{\alpha}\partial_u -\alpha\alpha_v\partial_v,\\
\nabla_{\partial_u}\partial_v = \nabla_{\partial_v}\partial_u = \tan(\theta)\mu_2\partial_u ,\\
\nabla_{\partial_v}\partial_v = 0.
\end{gather*}
We can also calculate the normal connection $\nabla^{\perp}$ of $M^2$ using (\ref{norLevi1}):
\begin{gather*}
\nabla^{\perp}_{\partial_{u}}\xi_1 = \nabla^{\perp}_{\partial_{u}}\xi_2 = 0,\\
\nabla^{\perp}_{\partial_{v}}\xi_1 =  -\tan(\theta)\mu_1\xi_2,\\
\nabla^{\perp}_{\partial_{v}}\xi_2 = \tan(\theta)\mu_1\xi_1.
\end{gather*}
The Codazzi equation gives us now that
\begin{gather}
(\mu_1)_u = 0,\label{Co1}\\
(\mu_2)_v = -\mu_2^2\tan(\theta) - \cos(\theta)\sin(\theta)c_2\label{Co2}.
\end{gather}
We immediately see that $\mu_1$ is a function that depends only on $v$. We solve now equations (\ref{int1}) and (\ref{Co2}). From equation (\ref{Co2}) we see that $\mu_2$ must satisfy the following PDE:
\begin{equation*}
(\mu_2)_v = -\tan(\theta)(c_2 \cos^2(\theta) + \mu_2^2).
\end{equation*}
By integration we obtain that $\mu_2$ must be equal to
\begin{equation*}
\begin{cases}
-\sqrt{c_2}\cos(\theta)\tan(\sqrt{c_2}\sin(\theta)v + C(u))&\textrm{ if } c_2 > 0,\\
 0 \: \textrm{ or }\: \frac{1}{\tan(\theta)v + C(u)} & \textrm{ if }\: c_2 = 0,\\
\pm\sqrt{-c_2}\cos(\theta)  \: \textrm{ or }\: \sqrt{-c_2}\cos(\theta)\tanh(\sqrt{-c_2}\sin(\theta)v + C(u))& \textrm{ if }\: c_2 < 0,
\end{cases}
\end{equation*}
where $C$ is some function depending on $u$.
Now, solving (\ref{int1}) we see that $\alpha$ equals
\begin{equation*}
\begin{cases}
D(u)\cos(\sqrt{c_2}\sin(\theta)v + C(u))& \textrm{ if }\: c_2 > 0,\\
D(u) \:\textrm{ or }\: D(u)(\tan(\theta)v + C(u))& \textrm{ if }\: c_2 = 0,\\
D(u)\exp{(\pm\sqrt{-c_2}\sin(\theta)v)}\: \textrm{ or }\: D(u)\cosh(\sqrt{-c_2}\sin(\theta)v + C(u))& \textrm{ if }\: c_2 < 0,
\end{cases}
\end{equation*}
where $D$ is some strictly positive function depending on $u$.

We will only consider the case for which $c_2 > 0$. The other cases can be treated analogously and the results of the other cases are stated in Theorem \ref{Th:const1}. So we can consider $M^2(c_1)\times M^2(c_2)$ as a submanifold of $\mathbb{E}^5,\mathbb{E}^6_1$ or $\mathbb{E}^6$ of codimension $1$ or $2$ and denote by $D$ the connection of $\mathbb{E}^5,\mathbb{E}^6_1$ or $\mathbb{E}^6$. Hence $M^2$ is an immersed surface in $\mathbb{E}^5,\mathbb{E}^6_1$ or $\mathbb{E}^6$. Remark now that $\xi_1$, $\xi_2$, which are tangent to $M^2(c_1)\times M^2(c_2)$, and $\bar{\xi} = (0,0,0,\psi_4,\psi_5,\psi_6)$ are normals of $M^2$ in $\mathbb{E}^5$ if $c_1 = 0$  and that $\xi_1,\xi_2,\widetilde{\xi} = (\psi_1,\psi_2,\psi_3,0,0,0)$ and $\bar{\xi}$ are normals of $M^2$ in $\mathbb{E}^6_1$ or $\mathbb{E}^6$ if $c_1 \neq 0$. Moreover we have that $F\xi_1 = \xi_1$ and hence $\xi_1$ is parallel to the first component of $M^2(c_1)\times M^2(c_2)$. One can verify that we have for every $X \in T_pM^2$,
\begin{equation}\label{Euc1}
D_X\widetilde{\xi} = \left(\frac{I + F}{2}\right)X = \left(\frac{I + f}{2}\right)X + \frac{hX}{2}
\end{equation}
and
\begin{equation}\label{Euc2}
D_X\bar{\xi} = \left(\frac{I - F}{2}\right)X = \left(\frac{I - f}{2}\right)X - \frac{hX}{2},
\end{equation}
where $F$ is the product structure of $M^2(c_1)\times M^2(c_2)$. Moreover the formulas of Gauss and Weingarten give that:
\begin{gather}
D_X Y = \nabla_X Y + \sigma(X,Y) - \frac{c_1}{2}g(\left(\frac{I + f}{2}\right)X,Y)\widetilde{\xi} - \frac{c_2}{2}g(\left(\frac{I - f}{2}\right)X,Y)\bar{\xi},\label{GaussD}\\
D_X\xi_1 = -S_{\xi_1}X + \nabla^{\perp}_X\xi_1,\label{WeinD1}\\
D_X\xi_2 = -S_{\xi_2}X + \nabla^{\perp}_X\xi_2 -\frac{c_1}{2}\widetilde{g}(hX,\xi_2)\widetilde{\xi} + \frac{c_2}{2}\widetilde{g}(hX,\xi_2)\bar{\xi}\label{WeinD2}.
\end{gather}
In the following we will consider the case for which $c_1 > 0$. The case for which $c_1 \leq 0 $ can be treated analogously.
Since $\partial_u = \alpha e_1$ we find using equation (\ref{Euc1}) that
\begin{equation*}
(\psi_1,\psi_2,\psi_3,0,0,0)_u = D_{\partial_u}\widetilde{\xi} = 0,
\end{equation*}
and hence we obtain that $\psi_i(u,v) = \widetilde{f}_i(v)$ for $i = 1,\dots,3$. Analogously we find that
\begin{gather*}
(\xi_2)_i = \tan(\theta)(\psi_i)_v \hbox{ for }i = 1,\dots ,3;\\
(\xi_2)_{j} = -\cot(\theta)(\psi_j)_v\hbox{ for }j = 4,\dots ,6.
\end{gather*}
We now use the formula of Gauss and the previous equations to find that
\begin{gather}
(\psi_{j})_{uu} = \frac{\alpha_u}{\alpha}(\psi_{j})_u - \alpha\alpha_v(\psi_{j})_v - \cot(\theta)\mu_2\alpha^2(\psi_{j})_v - c_2\alpha^2\psi_{j},\\
(\psi_{j})_{uv} = \frac{\alpha_v}{\alpha}(\psi_j)_{u} = \tan(\theta)\mu_2(\psi_j)_{u},\label{Gausseq1}\\
(\psi_j)_{vv} = -c_2\sin^2(\theta)\psi_j\label{Gausseq2}
\end{gather}
for $j = 4,5,6$.
Integrating equation (\ref{Gausseq1}), we find that
\begin{equation*}
(\psi_j)_{u} = \cos(\sqrt{c_2}v + C(u))H_j(u)
\end{equation*}
and hence we obtain that
\begin{equation*}
\psi_j = \int_{u_0}^{u}\cos(\sqrt{c_2}v + C(\tau))H_j(\tau)d\tau + I_{j}(v)
\end{equation*}
for $j = 4,5,6$ and with $H_j$ and $I_j$ arbitrary functions. Moreover, using equation (\ref{Gausseq2}) we find that the functions $I_j$ must satisfy
\begin{equation*}
I_j(v) = K_j\cos(\sqrt{c_2}\sin{\theta}v) + L_j\sin(\sqrt{c_2}\sin(\theta)v),
\end{equation*}
where $K_j$ and $L_j$ are constant. We summarize the previous and see that our immersion $\psi$ is given by
\begin{multline*}
\psi = (\widetilde{f}_1(v),\widetilde{f}_2(v),\widetilde{f}_3(v),\\
\left(K_4 +  \int_{u_0}^{u}H_4(\tau)\cos(C(\tau))d\tau
\right)\cos(\sqrt{c_2}\sin(\theta)v)\\
 + \left(L_4 -  \int_{u_0}^{u}H_4(\tau)\sin(C(\tau))d\tau
\right)\sin(\sqrt{c_2}\sin(\theta)v),\dots).
\end{multline*}
We define now the functions
\begin{gather*}
\bar{f}_j(u) = K_j +  \int_{u_0}^{u}H_j(\tau)\cos(C(\tau))d\tau,\\
\bar{g}_j(u) = L_j -  \int_{u_0}^{u}H_j(\tau)\sin(C(\tau))d\tau.
\end{gather*}
We use now some conditions to find a relation between $\bar{f}(u) = (\bar{f}_1(u),\bar{f}_2(u),\bar{f}_3(u))$ and $\bar{g}(u) = (\bar{g}_1(u),\bar{g}_2(u),\bar{g}_3(u))$:
\begin{gather*}
g(\psi_u,\psi_u) = \alpha^2,\:  g(\psi_v,\psi_v) = 1,\: g(\psi_u,\psi_v) = 0,\\
g(\xi_1,\psi_u) = 0,\:  g(\xi_1,\psi_v) = 0,\:  g(\xi_1,\xi_1) = 1,\\
g(\xi_2,\psi_u) = 0,\: g(\xi_2,\psi_v) = 0,\: g(\xi_2,\xi_2) = 1,\\
g(\widetilde{\xi},\psi_u) = 0,\: g(\widetilde{\xi},\psi_v) = 0,\: g(\widetilde{\xi},\widetilde{\xi}) = \frac{1}{c_1},\\
g(\bar{\xi},\psi_u) = 0,\: g(\bar{\xi},\psi_v) = 0,\: g(\bar{\xi},\bar{\xi}) = \frac{1}{c_2},\\
g(\xi_1,\xi_2) = 0,\:g(\xi_1,\widetilde{\xi}) = 0,\: g(\xi_2,\widetilde{\xi}) = 0,\: g(\xi_1,\bar{\xi}) = 0,\: g(\xi_2,\bar{\xi}) = 0,
\end{gather*}
which are equivalent to
\begin{gather*}
\sum_{i = 1}^3\widetilde{f}_i^2 = \frac{1}{c_1},\\
\sum_{j = 1}^3\bar{f}_j^2 = \frac{1}{c_2}, \sum_{j = 1}^3\bar{g}_j^2 = \frac{1}{c_2},\\
\sum_{j = 1}^3\bar{f}_j\bar{g}_j = 0,\sum_{j = 1}^3\bar{f}'_j\bar{g}_j = 0,\\
\sum_{i = 1}^3(\widetilde{f}'_i)^2= \cos^2(\theta),
\end{gather*}
\begin{multline}\label{geodcurv1}
\sum_{j = 1}^3(\bar{f}_j')^2\cos^2(\sqrt{c_2}\sin(\theta)v) + (\bar{g}_j')^2\sin^2(\sqrt{c_2}\sin(\theta)v)  + \\
2\bar{f}_j'\bar{g}_j'\cos(\sqrt{c_2}\sin(\theta)v)\sin(\sqrt{c_2}\sin(\theta)v) =  D^2(u)\cos^2(\sqrt{c_2}\sin(\theta)v + C(u)).
\end{multline}
From the above equations we see that $\bar{f}(u) = (\bar{f}_1(u),\bar{f}_2(u),\bar{f}_3(u))$ and $\bar{g}(u) = (\bar{g}_1(u),\bar{g}_2(u),\bar{g}_3(u))$ are curves in $M^2(c_2)$. Moreover if we change the $u$-coordinate such that $\bar{f}$ is a unit speed curve, which corresponds to the fact that $D^2(u) = \sec^2(C(u))$, we see then from the previous equations that $\bar{g}$ is a curve in $M^2(c_2)$ that is perpendicular to the vectors $\bar{f}$ and $\bar{f}'$. Hence we obtain that $\bar{g} = \pm \bar{f}\times\bar{f}' $ and we can choose that $\bar{g} =  \bar{f}\times\bar{f}' $. The immersion $\psi$ is then given by
\begin{equation*}
\psi(u,v) = (\widetilde{f}(v),\cos(\sqrt{c_2}\sin(\theta)v)\bar{f}(u) + \sin(\sqrt{c_2}\sin(\theta)v)\bar{f}(u)\times \bar{f}'(u)).
\end{equation*}
Let us remark that since $\overline{g} = \overline{f}\times \overline{f}'$, we obtain that $\overline{g}' = \frac{1}{\sqrt{c_2}}(J\overline{f}')' = -\frac{\overline{\kappa}}{\sqrt{c_2}}\overline{f}'$ and hence we have $\overline{f}'\cdot \overline{g}' = -\frac{\overline{\kappa}}{\sqrt{c_2}}$. Using equation (\ref{geodcurv1}), we obtain that $\frac{\overline{\kappa}}{\sqrt{c_2}} = \tan(C(u))$.
\end{proof}

The case for which $\lambda_1 = \cos(2\theta)$ and $\lambda_2 = 1$ can be treated analogously as the previous case. We summarize this case in the next theorem:

\begin{Th}\label{Th:constant2}
A surface $M^2$ isometrically immersed in $M^2(c_1)\times M^2(c_2)$ is a constant angle surface with angles $0$ and $\theta$ if and only if the immersion $\psi$ is locally given by
\begin{equation*}
\psi(u,v) = (\cos(\sqrt{c_1}\cos(\theta)u)\widetilde{f}(v) + \sin(\sqrt{c_1}\cos\theta)u)\widetilde{f}(v)\times\widetilde{f}'(v),\bar{f}(u)) \hbox{ if }c_1 > 0,
\end{equation*}
where $\bar{f}$ is a curve in $M^2(c_2)$ of constant speed $\sin(\theta)$ and $\widetilde{f}$ is a unit speed curve in $M^2(c_1)$; by
\begin{equation*}
\psi(u,v) = (\cosh(\sqrt{-c_1}\cos(\theta)u)\widetilde{f}(v) + \sinh(\sqrt{-c_1}\cos(\theta)u)\widetilde{f}(v)\boxtimes\ \widetilde{f}'(v),\bar{f}(u))\hbox{ if }c_1<0
\end{equation*}
where $\bar{f}$ is a curve in $M^2(c_2)$ of constant speed $\sin(\theta)$ and $\widetilde{f}$ is a unit speed curve in $M^2(c_1)$; or by
\begin{equation}\label{geod4}
\psi(u,v) = (\cos(\theta)u,v,\bar{f}(u)) \: or \: \psi(u,v) = (u\cos(\theta)\widetilde{f}(v) + \widetilde{g}(v),\bar{f}(u)) \hbox{ if } c_1 = 0,
\end{equation}
where $\bar{f}$ is a curve in $M^2(c_2)$ of constant speed $\sin(\theta)$, $\widetilde{f}(v) = (\cos(v),\sin(v))$ and $\widetilde{g}'(v) = -\sin(\theta)C(v)$ $(-\sin(v),\cos(v))$, where $C$ is a function on an interval $I$.
\end{Th}

We consider now the case for which $\lambda_1,\lambda_2 \in (-1,1)$ and $\lambda_2 - \lambda_1 \geq 0$ and show the following theorem.
\begin{Th}
Let $M^2$ be a constant angle surface with $\theta_1,\theta_2 \in (0,\frac{\pi}{2})$. Then there are $2$ possibilities:
\begin{enumerate}
  \item $M^2$ is an open part of the surfaces parameterized by
  \begin{gather*}
  \begin{split}(\cos(\sqrt{\frac{c_1c_2}{c_1 + c_2}}u)\widetilde{f}(v) + \sin(\sqrt{\frac{c_1c_2}{c_1 + c_2}}u)\frac{1}{\cos(\theta)}\widetilde{f}(v)\times \widetilde{f}'(v);\\
  \cos(\sqrt{\frac{c_1c_2}{c_1 + c_2}}u)\bar{f}(v) + \sin(\sqrt{\frac{c_1c_2}{c_1 + c_2}}u)\frac{1}{\sin(\theta)}\bar{f}(v)\times \bar{f}'(v))
  \end{split} \text{ if $c_1,c_2 > 0$},\\
  \begin{split}
  (\cosh(\sqrt{-\frac{c_1c_2}{c_1 + c_2}}u)\widetilde{f}(v) + \sinh(\sqrt{-\frac{c_1c_2}{c_1 + c_2}}u)\frac{1}{\cos(\theta)}\widetilde{f}(v)\boxtimes \widetilde{f}'(v);\\
  \cosh(\sqrt{-\frac{c_1c_2}{c_1 + c_2}}u)\bar{f}(v) + \sinh(\sqrt{-\frac{c_1c_2}{c_1 + c_2}}u)\frac{1}{\sin(\theta)}\bar{f}(v)\boxtimes \bar{f}'(v))\end{split} \text{ if $c_1,c_2 < 0$},
  \end{gather*}
where $\theta$ is equal to $\theta_1$ or $\theta_2$ and $\theta_2$ or $\theta_1$ is a real number in $(0,\frac{\pi}{2})$ such that $\cos^2(\theta_2)$ or $\cos^2(\theta_1)$ is equal to $\frac{c_2}{c_1 + c_2}$ and $\widetilde{f}$ is a curve in $M^2(c_1)$ of constant speed $\cos(\theta)$ and geodesic curvature $\widetilde{\kappa}$ and $\bar{f}$ is a curve in $M^2(c_2)$ of constant speed $\sin(\theta)$ and geodesic curvature $\overline{\kappa}$, such that $\frac{\widetilde{\kappa}}{\sqrt{|c_1|}} = \frac{\overline{\kappa}}{\sqrt{|c_2|}}$,
  \item a constant angle surface in $M^2(c_1)\times M^2(c_2)$ given by Proposition $\ref{Exis1},\ref{Exis2}$ or $\ref{Exis3}$.
\end{enumerate}
\end{Th}
\begin{proof}
After a straight-forward computation, one can deduce that the surfaces listed in the theorem are constant angle surfaces in $M^2(c_1)\times M^2(c_2)$.
Conversely, let us assume that $M^2$ is a constant angle surface in $M^2(c_1)\times M^2(c_2)$ with angles $\theta_1,\theta_2 \in (0,\frac{\pi}{2})$. Suppose first that $\theta_1 = \theta_2$. Then $M^2$ is a totally geodesic surface in $\Mcc$ and $\lambda_1 = \lambda_2 = \frac{c_2 - c_1}{c_1 + c_2}$. Hence $M^2$ is locally congruent to (\ref{geod1}) and (\ref{geod2}). So we are in the special case of case $1$ of the theorem. Let us now suppose that $\theta_1 \neq \theta_2$.  Then there is an adapted orthonormal frame $\{e_1,e_2,\xi_1,\xi_2\}$ such that $fe_i = \cos(2\theta_i)e_i$ and $t\xi_i = -\cos(2\theta_i)\xi_i$, where $\cos(2\theta_i) = \lambda_i$, for $i=1,2$. Moreover we have that the shape operators $S_{\xi_1}$ and $S_{\xi_2}$ have the same form as (\ref{shape}) with respect to $\{e_1,e_2\}$. Moreover the Levi-Civita connection is given by
\begin{gather}\label{Levi5}
\nabla_{e_1}e_1 = \frac{\sin(\theta_2)\cos(\theta_2)\mu_2}{\cos^2(\theta_1) - \cos^2(\theta_2)}e_2,\\
\nabla_{e_1}e_2 = \frac{\sin(\theta_2)\cos(\theta_2)\mu_2}{\cos^2(\theta_2) - \cos^2(\theta_1)}e_1,\\
\nabla_{e_2}e_1 = \frac{\sin(\theta_1)\cos(\theta_1)\mu_1}{\cos^2(\theta_1) - \cos^2(\theta_2)}e_2,\\
\nabla_{e_2}e_2 = \frac{\sin(\theta_1)\cos(\theta_1)\mu_1}{\cos^2(\theta_2) - \cos^2(\theta_1)}e_1.
\end{gather}
We also know the normal connection $\nabla^{\perp}$ of $M^2$ in $M^2(c_1)\times M^2(c_2)$:
\begin{gather*}
\nabla^{\perp}_{e_1}\xi_1 = \frac{\sin(\theta_1)\cos(\theta_1)\mu_2}{\cos^2(\theta_1) - \cos^2(\theta_2)}\xi_2 ,\\
\nabla^{\perp}_{e_1}\xi_2 = \frac{\sin(\theta_1)\cos(\theta_1)\mu_2}{\cos^2(\theta_2) - \cos^2(\theta_1)}\xi_1,\\
\nabla^{\perp}_{e_2}\xi_1 = \frac{\sin(\theta_2)\cos(\theta_2)\mu_1}{\cos^2(\theta_1) - \cos^2(\theta_2)}\xi_2,\\
\nabla^{\perp}_{e_2}\xi_2 = \frac{\sin(\theta_2)\cos(\theta_2)\mu_1}{\cos^2(\theta_2) - \cos^2(\theta_1)}\xi_1.
\end{gather*}
Using the expressions for the Levi-Civita connection and the normal connection, we find that the Codazzi equations are given by
\begin{gather}
e_1[\mu_1] + \frac{\sin(\theta_1)\cos(\theta_1)}{\cos^2(\theta_1) - \cos^2(\theta_2)}\mu_1^2 = (c_1\cos^2(\theta_2) - c_2\sin^2(\theta_2))\sin(\theta_1)\cos(\theta_1),\label{Codazzi1}\\
e_2[\mu_2] + \frac{\sin(\theta_2)\cos(\theta_2)}{\cos^2(\theta_2) - \cos^2(\theta_1)}\mu_2^2 = (c_1\cos^2(\theta_1) - c_2\sin^2(\theta_1))\sin(\theta_2)\cos(\theta_2)\label{Codazzi2}.
\end{gather}

\underline{Case 1: $\mu_1 = \mu_2 = 0$}. From the equations of Codazzi (\ref{Codazzi1}) and (\ref{Codazzi2}) we obtain that $c_1\cos^2(\theta_1)$\newline $ - c_2\sin^2(\theta_1)$
$ = c_1\cos^2(\theta_2) - c_2\sin^2(\theta_2) = 0$, because $\theta_1,\theta_2 \in (0,\frac{\pi}{2})$ and hence we obtain that $\cos(2\theta_1) = \cos(2\theta_2) = \frac{c_2 - c_1}{c_1 + c_2}$ with $c_1c_2 >0$. Since $\cos(2\theta_2) > \cos(2\theta_1)$ by assumption, we have a contradiction.

\underline{Case 2: $\mu_1 = 0,\mu_2 \neq 0$}. As before, using the equation of Codazzi, we obtain that $c_1\cos^2(\theta_2) - c_2\sin^2(\theta_2) = 0$ and hence we obtain that $\cos(2\theta_2) = \frac{c_2 - c_1}{c_1 + c_2}$ with $c_1c_2 > 0$. Denote in the following $\theta_1$ by $\theta$.
We will work out only the case for which $c_1 > 0$ and $c_2 > 0$. The other case can be treated analogously. From the expressions for the Levi-Civita connection, we find that $\nabla_{e_2}e_1 = \nabla_{e_2}e_2 = 0$. Let us take now coordinates on $M^2$ with $\partial_u  =\alpha e_1$ and $\partial_v = \beta e_2$. Using the condition $[\partial_u,\partial_v] = 0$ and the expressions for the Levi-Civita connection we find that
\begin{gather}
\alpha_v = \frac{\sqrt{c_1c_2}}{c_2\sin^2(\theta) - c_1\cos^2(\theta)}\alpha\beta\mu_2,\label{alpha1}\\
\beta_u = 0\label{beta1}.
\end{gather}
Equation (\ref{beta1}) implies that, after a change of the $u$-coordinate, we can assume that $\beta = 1$ and hence the metric takes the form
\begin{equation*}
g = \alpha^2du^2 + dv^2,
\end{equation*}
and so the Levi-Civita connection becomes:
\begin{gather*}
\nabla_{\partial_u}\partial_u = \frac{\alpha_u}{\alpha}\partial_u - \alpha\alpha_v\partial_v,\\
\nabla_{\partial_u}\partial_v = \nabla_{\partial_v}\partial_u = \frac{\sqrt{c_1c_2}}{c_2\sin^2(\theta) - c_1\cos^2(\theta)}\mu_2\partial_u,\\
\nabla_{\partial_v}\partial_v = 0.
\end{gather*}
The equation of Codazzi (\ref{Codazzi2}) can now be rewritten as
\begin{equation}\label{CodazziCase2}
(\mu_2)_v = \frac{c_1c_2}{c_1\cos^2(\theta) - c_2\sin^2(\theta)}\left(\frac{(c_1\cos^2(\theta) - c_2\sin^2(\theta))^2}{c_1 + c_2} + \mu_2^2\right).
\end{equation}
Integrating equations (\ref{alpha1}) and (\ref{CodazziCase2}) we find
\begin{gather*}
\mu_2 = \frac{c_1\cos^2(\theta) - c_2\sin^2(\theta)}{\sqrt{c_1 + c_2}}\tan(\sqrt{\frac{c_1c_2}{c_1 + c_2}}v + C(u)),\\
\alpha =D(u)\cos(\sqrt{\frac{c_1c_2}{c_1 + c_2}}v + C(u)).
\end{gather*}

We consider now the surface $M^2$ as a codimension 4 immersed surface in the Euclidean space $\mathbb{E}^6$. By $D$ we will denote the Euclidean connection. We remark that $\xi_1,\xi_2,\widetilde{\xi} = (\psi_1,\psi_2,\psi_3,0,0,0)$ and $\overline{\xi} = (0,0,0,\psi_4,\psi_5,\psi_6)$ are normals of $M^2$ in $\mathbb{E}^6$. We still have that the equations (\ref{Euc1}) and (\ref{Euc2}) hold. Moreover the equations of Gauss and Weingarten are given by
\begin{gather*}
D_X Y = \nabla_X Y + \sigma(X,Y) - \frac{c_1}{2}g(\left(\frac{I + f}{2}\right)X,Y)\widetilde{\xi} - \frac{c_2}{2}g(\left(\frac{I - f}{2}\right)X,Y)\bar{\xi},
,\\
D_X\xi_1 = -S_{\xi_1}X + \nabla^{\perp}_X\xi_1 -\frac{c_1}{2}\widetilde{g}(hX,\xi_1) + \frac{c_2}{2}\widetilde{g}(hX,\xi_1)
,\\
D_X\xi_2 = -S_{\xi_2}X + \nabla^{\perp}_X\xi_2 -\frac{c_1}{2}\widetilde{g}(hX,\xi_2) + \frac{c_2}{2}\widetilde{g}(hX,\xi_2)
.
\end{gather*}
Now applying the formula Gauss and the previous equations we find
\begin{gather}
D_{\partial_u}\partial_u  =  \frac{\alpha_u}{\alpha}\partial_u - \alpha\alpha_v\partial_v + \mu_2\alpha^2\xi_2 - c_1\cos^2(\theta_1)\alpha^2\widetilde{\xi} - c_2\sin^2(\theta_1)\alpha^2\bar{\xi},\\
D_{\partial_u}\partial_v = D_{\partial_v}\partial_u = -\sqrt{\frac{c_1c_2}{c_1 + c_2}}\tan(\sqrt{\frac{c_1c_2}{c_1 + c_2}}v + C(u))\partial_u,\label{GD1Case2}\\
D_{\partial_v}\partial_v = -\frac{c_1c_2}{c_1 + c_2}(\widetilde{\xi} + \bar{\xi})\label{GD2Case2},
\end{gather}
where
\begin{gather*}
(\xi_2)_i = \sqrt{\frac{c_1}{c_2}}(\psi_i)_{v}\: \textrm{ for }\: i= 1,2,3,\\
(\xi_2)_j = -\sqrt{\frac{c_1}{c_2}}(\psi_j)_{v}\: \textrm{ for }\: j= 4,5,6.\\
\end{gather*}

Integrating the last two formulas of Gauss, i.e. (\ref{GD1Case2}) and (\ref{GD2Case2}), we obtain analogously as before that
\begin{equation*}
\psi(u,v) = (\cos(\sqrt{\frac{c_1c_2}{c_1 + c_2}}u)\widetilde{f}(v) + \sin(\sqrt{\frac{c_1c_2}{c_1 + c_2}}u)\widetilde{g}(v),\cos(\sqrt{\frac{c_1c_2}{c_1 + c_2}}u)\bar{f}(v) + \sin(\sqrt{\frac{c_1c_2}{c_1 + c_2}}u)\bar{g}(v)),
\end{equation*}
where $\widetilde{f}(v) = (\widetilde{f}_1(v),\widetilde{f}_2(v),\widetilde{f}_3(v)), \widetilde{g}(v) = (\widetilde{g}_1(v),\widetilde{g}_2(v),\widetilde{g}_3(v)), \bar{f}(v) = (\bar{f}_1(v),\bar{f}_2(v),\bar{f}_3(v))$ and $\bar{g}(v) = (\bar{g}_1(v),\bar{g}_2(v),\bar{g}_3(v))$ and
\begin{gather*}
\widetilde{f}_i(v) = \widetilde{K}_i +  \int_{v_0}^{v}\widetilde{H}_i(\tau)\cos(C(\tau))d\tau,\\
\widetilde{g}_i(v) = \widetilde{L}_i -  \int_{v_0}^{v}\widetilde{H}_i(\tau)\sin(C(\tau))d\tau,\\
\overline{f}_j(v) =\overline{K}_j +  \int_{v_0}^{v}\overline{H}_j(\tau)\cos(C(\tau))d\tau,\\
\overline{g}_j(v) = \overline{L}_j -  \int_{v_0}^{v}\overline{H}_j(\tau)\sin(C(\tau))d\tau,
\end{gather*}
for $i,j =1,\dots,3$. Moreover we have the following equations
\begin{gather*}
g(\psi_u,\psi_u) = \alpha^2,\:  g(\psi_v,\psi_v) = 1,\: g(\psi_u,\psi_v) = 0,\\
g(\xi_1,\psi_u) = 0,\:  g(\xi_1,\psi_v) = 0,\:  g(\xi_1,\xi_1) = 1,\\
g(\xi_2,\psi_u) = 0,\: g(\xi_2,\psi_v) = 0,\: g(\xi_2,\xi_2) = 1,\\
g(\widetilde{\xi},\psi_u) = 0,\: g(\widetilde{\xi},\psi_v) = 0,\: g(\widetilde{\xi},\widetilde{\xi}) = \frac{1}{c_1},\\
g(\bar{\xi},\psi_u) = 0,\: g(\bar{\xi},\psi_v) = 0,\: g(\bar{\xi},\bar{\xi}) = \frac{1}{c_2},\\
g(\xi_1,\xi_2) = 0,\:g(\xi_1,\widetilde{\xi}) = 0,\: g(\xi_2,\widetilde{\xi}) = 0,\: g(\xi_1,\bar{\xi}) = 0,\: g(\xi_2,\bar{\xi}) = 0,
\end{gather*}
which are equivalent to
\begin{gather*}
\sum_{i = 1}^3\widetilde{f}_i^2 = \frac{1}{c_1}  = \sum_{i = 1}^3\widetilde{g}_i^2,\\
\sum_{j = 1}^3\widetilde{f}_j^2 = \frac{1}{c_2} =\sum_{j = 1}^3\bar{g}_j^2,\\
\sum_{i = 1}^3\widetilde{f}_i\widetilde{g}_i = 0 = \sum_{i = 1}^3\widetilde{f}'_i\widetilde{g}_i,\\
\sum_{j = 1}^3\bar{f}_j\bar{g}_j = 0 = \sum_{j = 1}^3\bar{f}'_j\bar{g}_j ,
\end{gather*}
\begin{multline*}
\sum_{i = 1}^3\left((\widetilde{f}_i')^2\cos^2(\sqrt{\frac{c_1c_2}{c_1 + c_2}}v) + (\widetilde{g}_i')^2\sin^2(\sqrt{\frac{c_1c_2}{c_1 + c_2}}v)  + 2\widetilde{f}_i'\widetilde{g}_i'\cos(\sqrt{\frac{c_1c_2}{c_1 + c_2}}v)\sin(\sqrt{\frac{c_1c_2}{c_1 + c_2}}v)\right)\\
 =  \cos^2(\theta)D^2(u)\cos^2(\sqrt{\frac{c_1c_2}{c_1 + c_2}}v + C(u)),
\end{multline*}
\begin{multline*}
\sum_{j = 1}^3\left((\bar{f}_j')^2\cos^2(\sqrt{\frac{c_1c_2}{c_1 + c_2}}v) + (\bar{g}_j')^2\sin^2(\sqrt{\frac{c_1c_2}{c_1 + c_2}}v)  + 2\bar{f}_j'\bar{g}_j'\cos(\sqrt{\frac{c_1c_2}{c_1 + c_2}}v)\sin(\sqrt{\frac{c_1c_2}{c_1 + c_2}}v)\right)\\
 =  \sin^2(\theta)D^2(u)\cos^2(\sqrt{c_2}\sin(\theta)v + C(u)).
\end{multline*}
We obtain from the above equations that $\widetilde{f}$ and $\widetilde{g}$ are curves in $M^2(c_1)$, that $\bar{f}$ and $\bar{g}$ are curves in $M^2(c_2)$. If we change the $v$-coordinate such that $\widetilde{f}$ and $\bar{f}$ have constant speed $\cos(\theta)$ and $\sin(\theta)$, which corresponds to the fact that $D^2(v) = \sec^2(C(v))$, we see then from the previous equations that $\widetilde{g} = \pm\frac{1}{\cos(\theta)}\widetilde{f} \times \widetilde{f}'$ and $\bar{g} = \pm\frac{1}{\sin(\theta)}\bar{f} \times \bar{f}'$ and we can choose $\widetilde{g} = \frac{1}{\cos(\theta)}\widetilde{f} \times \widetilde{f}'$ and $\bar{g} = \frac{1}{\sin(\theta)}\bar{f} \times \bar{f}'$. From the last two equations we also deduce that $\frac{\widetilde{f}'\cdot\widetilde{g}'}{\cos^2(\theta)} = \frac{\bar{f}'\cdot\bar{g}'}{\sin^2(\theta)}$ and $\frac{\widetilde{g}'\cdot \widetilde{g}'}{\cos^2(\theta)} = \frac{\bar{g}'\cdot\bar{g}'}{\sin^2(\theta)}$. This is equivalent to $\frac{\widetilde{\kappa}}{\sqrt{c_1}} = \frac{\overline{\kappa}}{\sqrt{c_2}}$, where $\widetilde{\kappa}$ and $\overline{\kappa}$ are the geodesic curvatures of respectively $\widetilde{f}$ and $\overline{f}$. So we obtain the first case of the theorem.

\underline{Case 3: $\mu_1 \neq 0,\mu_2 \neq 0$}. Let $(u,v)$ be coordinates on $M^2$ such that $\partial_u = \alpha e_1 $ and $\partial_v  = \beta e_2$. From the expression of the Levi-Civita connection, i.e. equation (\ref{Levi5}) and the condition $[\partial_u,\partial_v] = 0$, we obtain
\begin{gather}
a_2\mu_2 = \frac{\alpha_v}{\alpha\beta},\label{connection1}\\
a_1\mu_1 = \frac{\beta_u}{\alpha\beta}\label{connection2},
\end{gather}
where $a_1$ and $a_2$ are constants as in Proposition $\ref{Exis1}$. Using the previous equations, we can rewrite the equations of Codazzi (\ref{Codazzi1}) and (\ref{Codazzi2}) as follows
\begin{gather*}
(\alpha^2(\mu_2^2 + A_2))_v = 0,\\
(\beta^2(\mu_1^2 + A_1))_u = 0,
\end{gather*}
where $A_1$ and $A_2$ are constants as in Proposition $\ref{Exis1}$ and hence we obtain that $\alpha^2(\mu_2^2 + A_2) = C_2(u)$ and $\beta^2(\mu_1^2 + A_1) = C_1(v)$. We have to consider now several subcases.

\underline{Case 3.a.: $C_1\neq 0$, $C_2\neq 0$}. After a transformation of the $u$-coordinate and the $v$-coordinate we can suppose that $\alpha^2 = \frac{1}{\mu_2^2 + A_2}$ and $\beta^2 = \frac{1}{\mu_1^2 + A_1}$. Substituting this in equations (\ref{connection1}) and (\ref{connection2}) we obtain equations (\ref{exiscon1}). We can conclude that the isometric immersion $\psi$ is locally congruent  to the surface of Proposition $\ref{Exis1}$.

\underline{Case 3.b.: $C_1 = 0$, $C_2 \neq 0$}.  Since $C_1 = 0$, we have that $\mu_1^2 + A_1 = 0$. So we have that $A_1 < 0$, because $\mu_1 \neq 0$. We conclude that $\mu_1 = \pm\sqrt{-A_1}$ and without loss of generalization we can suppose that $\mu_1 = \sqrt{-A_1}$ . After a transformation of the $u$-coordinate, we can suppose that $\alpha^2 = \frac{1}{\mu_2^2 + A_2}$. Substituting the last two equations into equations (\ref{connection1}) and (\ref{connection2}) we obtain equations (\ref{exiscon2}). We can conclude that the isometric immersion $\psi$ is locally congruent  to the surface of Proposition $\ref{Exis2}$.

\underline{Case3.b.: $C_1 = C_2 = 0$}. Analogously as before we conclude that $\psi$ is locally congruent to the surface of Proposition $\ref{Exis3}$.
\end{proof}

We end this paper with a classification of the totally geodesic surfaces of $M^2(c_1)\times M^2(c_2)$. We have seen that a totally geodesic surface of $M^2(c_1)\times M^2(c_2)$ is a constant angle surfaces of $M^2(c_1)\times M^2(c_2)$. We will use the classification of the constant angle surfaces to give the classification of the totally geodesic surfaces of $M^2(c_1)\times M^2(c_2)$.

\begin{Th}\label{theoremgeod}
Let $M^2$ be a totally geodesic surface of $M^2(c_1)\times M^2(c_2)$. Then there are four possibilities
\begin{enumerate}
\item $M^2$ is locally congruent to the immersion given by (\ref{geod1}) or by (\ref{geod2});
\item $M^2$ is a product of two geodesic curves;
\item $M^2$ is an open part of $M^2(c_1)\times \{p_2\}$ or $\{p_1\}\times M^2(c_2)$;
\item $M^2$ is locally congruent to the first immersion of (\ref{geod3}), in which the curve $\widetilde{f}$ is a geodesic curve of $M^2(c_1)$ if $c_2 = 0$, or to the first immersion of (\ref{geod4}), in which the curve $\overline{f}$ is a  geodesic curve of $M^2(c_2)$ if $c_1 = 0$.
\end{enumerate}
\end{Th}
\begin{proof}
Proposition $\ref{geodconstant}$ tells us that a totally geodesic surface of $M^2(c_1)\times M^2(c_2)$ is a constant angle surface. Moreover in the proof of Proposition $\ref{geodconstant}$, we have seen that there are four possible situations. In the first case we have seen that the angle functions $\lambda_1$ and $\lambda_2$ are equal and have value $-\frac{b}{a}$, in which $a = \frac{c_1 + c_2}{4}$ , $b = \frac{c_1 - c_2}{4}$ and $c_1c_2 > 0$. We have classified these totally geodesic surfaces in Proposition $\ref{pro:geod}$ and showed that they are locally congruent to (\ref{geod1}) if $c_1,c_2 >0$ or to (\ref{geod2}) if $c_1,c_2 < 0$. The second case says that the angle functions are equal to $\pm 1$. If the angle functions have opposite sign then one can easily show that the surface is a Riemannian product of curves of $M^2(c_1)$ and $M^2(c_2)$ and that this surface is totally geodesic if and only if both curves are geodesic curves of $M^2(c_1)$ and $M^2(c_2)$. If both angle functions have the same sign then one can easily deduce that the surface is an open part of $M^2(c_1)\times \{p_2\}$ if the angle functions are equal to $1$ or an open part of $\{p_1\}\times M^2(c_2)$ if the angle functions are equal to $-1$. The third case in the proof tells us that one of the angle functions is $\pm 1$ and the other angle function is equal to $-\frac{b}{a}$. We show that in this case there exist only totally geodesic surfaces in the case that $c_1 = 0$ or $c_2 =0$. Suppose therefore that $c_1 \neq 0$ and $c_2 \neq 0$. Remark that $-\frac{b}{a} \neq \pm1$, because $c_1 \neq 0$ and $c_2 \neq 0$. In Propositions $\ref{cond1}$ and $\ref{cond2}$ we have showed that the curvature of the surface is $c_2\sin^2(\theta)$ if one of the angle functions is $-1$ or $c_1\cos^2(\theta)$ if one of the angle functions is $1$. So if the other angle function is equal to $-\frac{b}{a}$, then we obtain that in both cases the curvature is equal to $\frac{c_1c_2}{c_1 + c_2}$. But  the surface is totally geodesic and hence we obtain form equations (\ref{Levi1}) and (\ref{Levi2}) that the surface is flat. Finally we obtain that $\frac{c_1c_2}{c_1 + c_2} = 0$ and hence $c_1 = 0$ or $c_2 = 0$. This is of course a contradiction, because we have assumed that $c_1 \neq 0$ and $c_2 \neq 0$. Hence we obtain that in the third case $c_1 = 0$ or $c_2 = 0$ and so the angle functions in this case are $\pm 1$. This brings us back to the second case and hence we are finished. In the fourth case we have that one of the angle functions is equal to $1$ and the other angle function is a constant in $[-1,1]$ if $c_1 = 0$ or that one of the angle functions is equal to $-1$ and the other angle functions is a constant in $[-1,1]$ if $c_2 = 0$. We can suppose that the other angle function is a constant in $(-1,1)$. We can easily deduce from the classification theorems $\ref{Th:const1}$ and $\ref{Th:constant2}$, that $M^2$ is indeed locally congruent to the first immersion of (\ref{geod3}), in which the curve $\widetilde{f}$ is a geodesic curve of $M^2(c_1)$ if $c_2 = 0$, or to the first immersion of (\ref{geod4}), in which the curve $\overline{f}$ is a totally geodesic curve of $M^2(c_2)$ if $c_1 = 0$.
\end{proof}

\begin{Rm}
The special case when $c_1=c_2=2$ of Theorem \ref{theoremgeod} was also considered in the paper \cite{CL}
where totally geodesic surfaces in $Q^n$ (in particular, in $Q^2=S^2(2)\times S^2(2)$) were classified. In particular, they proved that a totally geodesic surface of $Q^n$ is one of the following three kinds:
\begin{enumerate}
\item a totally geodesic totally real surface,
\item a totally geodesic complex surface,
\item a totally geodesic surface of curvature 1/5 in $Q^n$ which is neither totally real nor complex. This case occurs only when $n\geq 3$.
\end{enumerate}
Since the third case doesn't occur for $n=2$, this is consistent with our results. It is interesting to remark that the third case was missing in \cite{CN1} and \cite{CN2}. This was remarked by S. Klein in \cite{Kle}, who didn't notice that the missing case did occur in the earlier paper \cite{CL}. The authors would like to thank B.-Y. Chen for drawing our attention to \cite{CL}.
\end{Rm}


\begin{thebibliography}{9999}

\bibitem{BDOR} P. Bayard, A. Di Scala, G. Ruiz-Hernandez, O. Osuna Castro, \emph{Surfaces in $\mathbb{R}^4$ with constant principal angles
with respect to a plane}, preprint 2011.

\bibitem{CU}
I. Castro, F. Urbano, \emph{Minimal Lagrangian surfaces in $\mathbb{S}^2\times\mathbb{S}^2$}, Comm. Anal. Geom. 15 (2007), 217–248.

\bibitem{CL}
B.-Y. Chen and H.-S. Lue, \emph{Differential geometry of $SO(n+2)/SO(2)\times SO(n)$, I}, Geom. Dedicata 4 (1975), 253-261.

\bibitem{CN1}
B.-Y. Chen and T. Nagano, \emph{Totally geodesic submanifolds of symmetric spaces, I},
Duke Math. J. 44 (1977), no. 4, 745--755.

\bibitem{CN2}
B.-Y. Chen and T. Nagano, \emph{Totally geodesic submanifolds of symmetric spaces, II}, Duke Math. J. 45 (1978), 405–425.

\bibitem{DFVV}
F. Dillen, J. Fastenakels, J. Van der Veken and L. Vrancken, \emph{Constant Angle
Surfaces in $\mathbb{S}^2\times \mathbb{R}$}, Monaths. Math., 152(2) (2007), 89–96.

\bibitem{DM}
F. Dillen, M.I. Munteanu, \emph{Constant angle surfaces in $\mathbb {H}^2\times\mathbb{R}$}, Bull. Braz. Math. Soc. (N.S.), 40(1) (2009), 85–97.

\bibitem{J}
C. Jordan, \emph{ Essais sur la g\'eom\'etrie \`a $n$ dimensions}, Bull. Soc. Math. France 3 (1875), 103-174


\bibitem{KS}
M. Kimura and K.Suizu, \emph{Fundamental theorems of Lagrangian surfaces in $S^2\times S^2$},  Osaka J. Math. 44 (2007), 829–850.

\bibitem{Kle}
S. Klein, \emph{Totally geodesic submanifolds of the complex quadric}, Differential Geom. Appl. 26 (2008), 79--96.


\bibitem{Kuh}
W. Kühnel,  Differential geometry. Curves—surfaces—manifolds, Student Mathematical Library, 16. American Mathematical Society, Providence, RI, 2002.

\bibitem{K}
D. Kowalczyk, \emph{Isometric immersions into product of space forms}, Geom. Dedicata, 151 (2011),  1-8.

\bibitem{MR}
W. Meeks, H. Rosenberg, \emph{Stable minimal surfaces in $ M^2\times\mathbb{R} $}, J. Differential Geom., 68 (2004), 515-534.

\bibitem{MR0482513}
Y.C. Wong, Linear geometry in Euclidean 4-space, SEAMS Monograph, No. 1. Southeast Asian Mathematical Society, Singapore, 1977.

\bibitem{R}
H. Rosenberg, \emph{Minimal surfaces in $ M^2\times\mathbb{R} $}, Illinois J. Math., 46 (2002),  1177-1195.

\bibitem{YK}
K. Yano and M. Kon, Structures on manifolds, Series in Pure Mathematics, 3. World Scientific Publishing Co.
\end{thebibliography}
\end{document}